\documentclass[12pt]{amsart}
\usepackage{etex} \usepackage{amssymb,amsthm,amsmath,amsfonts,stmaryrd,soul}
\makeatletter{}\usepackage[OT2,OT1]{fontenc}
\newcommand\cyr
{
\renewcommand\rmdefault{wncyr}
\renewcommand\sfdefault{wncyss}
\renewcommand\encodingdefault{OT2}
\normalfont
\selectfont
}
\DeclareTextFontCommand{\textcyr}{\cyr}

\usepackage[colorlinks,anchorcolor=black,citecolor=black,linkcolor=black,filecolor=black,urlcolor=black]{hyperref}

\usepackage{color}

\usepackage{pict2e}

\usepackage[margin=2cm,a4paper]{geometry}

\usepackage{tikz,pgfplots,pgfplotstable}

\usetikzlibrary{patterns}
\usetikzlibrary{backgrounds}

\usepackage{ifthen}
\newboolean{coloron}
\setboolean{coloron}{false}
\ifthenelse{\boolean{coloron}}
{
  \tikzset{    afill/.style={draw,color=red,fill=red},
    bfill/.style={draw,color=green,fill=green}
  }
  \tikzstyle{thline}=[thick]
  \tikzstyle{cline}=[color=red]
  \tikzstyle{aline}=[color=red]
  \tikzstyle{bline}=[color=green]
}
{
    \tikzset{      afill/.style={draw,pattern=crosshatch},
      bfill/.style={draw,pattern=crosshatch dots}
    }
    \tikzstyle{thline}=[thick]
    \tikzstyle{cline}=[densely dashdotted]
    \tikzstyle{aline}=[densely dashdotted]
    \tikzstyle{bline}=[densely dotted]
}

\usepackage{mathabx}
\newlength{\tmpl}
\newcommand{\boxstar}{
\setlength{\tmpl}{\fboxsep}
\setlength{\fboxsep}{0pt}
\,\fbox{$\star$}\,
\setlength{\fboxsep}{\tmpl}
}

\usepackage[extra]{tipa}
\usepackage{graphicx}
\usepackage{caption,subcaption}

\usepackage{enumerate}

\numberwithin{equation}{section}

\newcommand{\onetwo}[1]{\hat{1}_2^{#1}} 
\newcommand{\pispecial}[1]{\nu_{0#1}} 
\newcommand{\Comp}[1]{\mathcal{C}_{#1}} \newcommand{\CompEven}[1]{\mathcal{C}^e_{#1}} \newcommand{\CompOdd}[1]{\mathcal{C}^o_{#1}}

\DeclareMathOperator{\Moeb}{\textup{M{\"{o}}b}}
\DeclareMathOperator{\Cf}{\mathrm{Cf}}

\DeclareMathOperator{\Tr}{\mathrm{Tr}}

\let\Re\undefined
\let\Im\undefined
\DeclareMathOperator{\Re}{\mathrm{Re}}
\DeclareMathOperator{\Im}{\mathrm{Im}}

\newcommand{\Krewl}[1]{\mathpalette\harpleft{#1}}
\newcommand{\harpleftsign}{\scriptstyle\leftharpoonup}
\newcommand{\harpleft}[2]{  \ifx\displaystyle#1\doalign{$\harpleftsign$}{#1#2}\fi
  \ifx\textstyle#1\doalign{$\harpleftsign$}{#1#2}\fi
  \ifx\scriptstyle#1\doalign{\scalebox{.6}[.9]{$\harpleftsign$}}{#1#2}\fi
  \ifx\scriptscriptstyle#1\doalign{\scalebox{.5}[.8]{$\harpleftsign$}}{#1#2}\fi
}

\newcommand{\Krewr}[1]{\mathpalette\harpright{#1}} \newcommand{\harprightsign}{\scriptstyle\rightharpoonup}
\newcommand{\harpright}[2]{  \ifx\displaystyle#1\doalign{$\harprightsign$}{#1#2}\fi
  \ifx\textstyle#1\doalign{$\harprightsign$}{#1#2}\fi
  \ifx\scriptstyle#1\doalign{\scalebox{.6}[.9]{$\harprightsign$}}{#1#2}\fi
  \ifx\scriptscriptstyle#1\doalign{\scalebox{.5}[.8]{$\harprightsign$}}{#1#2}\fi
}

\newcommand{\doalign}[2]{ {\vbox{\offinterlineskip\ialign{\hfil##\hfil\cr#1\cr$#2$\cr}}}}

\newcommand{\ED}[1][{}]{E^{\mathcal D}_{#1}}

\def\R{{\mathbb R}}

\def\C{{\mathbb C}}
\def\IC{{\mathbb C}}
\def\IR{{\mathbb R}}

\def\IN{{\mathbb N}}
\def\N{{\mathbb N}}

\def\T{\mathcal{T}}
\def\bX{{\boldsymbol{X}}}
\def\<{\langle}
\def\>{\rangle}

\def\X{{\mathnormal X}}

\def\c{\underline c}

\def\Y{{\mathnormal Y}}

\def\A{\mathcal{A}}

\def\BL{{{\leftthreetimes}}_\pi}\def\BR{{{\rightthreetimes}}_\pi}

\def\SpPair{\sqcap} \def\nest{\textbf{N}}

\newcommand{\abs}[1]{
  \left\lvert
    #1
  \right\rvert}

 \DeclareMathOperator{\SG}{\mathfrak{S}} \DeclareMathOperator{\NC}{\mathit{NC}}  \DeclareMathOperator{\NCeven}{\mathit{NCE}} \DeclareMathOperator{\NCodd}{\mathit{NCE}^c} \DeclareMathOperator{\Rtrans}{\mathcal{C}} \newcommand{\Rtranseven}[1]{  {\mathcal{C}_{#1}^\mathrm{(even)}}} 

\newcommand{\FriCAS}{\texttt{FriCAS}}

\newtheorem{th-def}{Theorem-Definition}[section]
\newtheorem{theo}{Theorem}[section]
\newtheorem{lemm}[theo]{Lemma}

\newtheorem{prop}[theo]{Proposition}
\newtheorem{cor}[theo]{Corollary}
\theoremstyle{definition}

\newtheorem{defi}[theo]{Definition}

\newtheorem{Rem}[theo]{Remark}

\def\cvput#1[#2]{\pnode(#1,1){#1} \pscircle*(#1,1){.1} \rput(#1,.5){$#2$}}

\title{Sums of Commutators in Free Probability}
\author[Wiktor Ejsmont]{Wiktor Ejsmont}

\address[Wiktor Ejsmont]{ 
Instytut Matematyczny, Uniwersytet Wroc\l awski,\\ pl. Grunwaldzki 2/4, 50-384 Wroc\l aw, Poland}
\email{wiktor.ejsmont@gmial.com}
\author[Franz Lehner]{Franz Lehner}
\address[Franz Lehner]{Institut f\"ur Diskrete Mathematik,
Technische Universit\"at Graz,
Steyrergasse 30, 8010 Graz, Austria }
\email{lehner@math.tugraz.at}

\thanks{Supported by the Austrian Federal Ministry of Education, Science and
  Research and the Polish Ministry of Science and Higher Education, grants
  N$^{\textrm{os}}$ PL 08/2016 and PL 06/2018 and Wiktor
Ejsmont was supported by the Narodowe Centrum Nauki grant no. 2018/29/B/HS4/01420}


\begin{document}

\begin{abstract} 
We study the linear span of commutators of free random variables
and  show that these are the only quadratic forms which satisfy
the following equivalent properties:
\begin{itemize}
\item preservation free infinite divisibility,  
\item free and strong cancellation of odd cumulants,
\item symmetric distribution for any free family.
\end{itemize} 
The main combinatorial tool is an involution on non-crossing partitions. 
\end{abstract}

\date{
  \today}

\setlength{\parindent}{0pt}

\maketitle

\section{Introduction}
Free probability was introduced by Voiculescu 30 years ago
\cite{Voiculescu:1986,VoiculescuDykemaNica:1992} in order to solve
some problems in von Neumann algebras of free groups. It has developed into a
whole new field with numerous connections to 
different branches of mathematics like classical probability,
combinatorics and analysis, in particular random matrices
\cite{Voiculescu:1991}, noncrossing partitions
\cite{NicaSpeicher:1996} and operator algebras.
Free probability is considered the most developed branch 
of noncommutative probability and during its development far-reaching
analogies between classical and free probability emerged.
For example, there is a central limit theorem with the famous Wigner semicircle
law appearing in the limit, a corresponding Brownian motion, and  more
generally, one of the most striking features is the existence of 
the Bercovici-Pata bijection \cite{BercoviciPata:1999} between infinitely divisible distributions in the classical and the free world.

In the present article we continue our investigation of the distribution of quadratic forms
\cite{EjsmontLehner:2017}. The main result is a  characterization of quadratic
forms which preserve free infinite divisibility.
It was shown in \cite{ArizmendiHasebeSakuma:2013}  that the free commutator of
freely infinitely
divisible random variables is also freely infinitely divisible
and the authors ask whether there
are other noncommutative polynomials which preserve free infinite divisibility. 
In \cite{EjsmontLehner:2017} we showed that any quadratic form in free random
variables which exhibits the phenomenon of \emph{cancellation of odd cumulants},
i.e., whose distribution does not depend on the odd cumulants
of the distributions of the original variables, preserves free infinite
divisibility.
Examples are  the free commutator \cite{NicaSpeicher:1998}
and the free sample variance  \cite{EjsmontLehner:2017}.
Note that  the  cancellation phenomenon for the latter
only holds for free identically distributed families, while in the former
arbitrary free random variables can be inserted.

In the present paper we give a unified proof of these results.
In addition we record the observation that the
cancellation phenomenon for the commutator also occurs without
the freeness assumption, i.e., the remarkable phenomenon that the mixed odd cumulants
cancel
for sums of commutators of arbitrary noncommutative random variables.
Using these results  we introduce generalized tetilla laws.

\emph{Acknowledgements.}  
We thank the anonymous referee of the first version of this paper for numerous corrections and improvements.

\section{Preliminaries}
\label{sec:prelim}

\subsection{Basic Notation and Terminology}
A tracial noncommutative probability space is a pair $(\mathcal{A},\tau)$
where $\mathcal{A}$ is a von Neumann algebra, and  $\tau:\mathcal{A} \to
\mathrm{C}$ is a normal, faithful, tracial state, i.e., $\tau$ is linear and
continuous in the weak* topology, $\tau(\X \Y)=\tau(\Y \X)$,  $\tau(I)=1$,
$\tau(\X\X^*)\geq 0$ and $\tau(\X \X^{*}) = 0$ implies $\X = 0$ for all $\X,\Y
\in \mathcal{A}$.
The basic example of a noncommutative probability space is the algebra of
complex $N\times N$ matrices $M_N(\IC)$. The unique tracial state is
the normalized trace $\tau_N(A)=\frac{1}{N}\Tr(A)=\frac{1}{N}\sum A_{ii}$.

The elements $\X\in{\mathcal{A}}_{sa}$
are called (noncommutative) random variables; in the present paper
all random variables are assumed to be self-adjoint.
Given a noncommutative random variable $\X\in{\mathcal{A}}_{sa}$,
the spectral theorem provides a unique probability measure $\mu_X$ on 
$\R$ which encodes the distribution of $\X$ in the state $\tau$,
i.e., 
 $\tau( f(\X))=\int_\R f(\lambda)\,d\mu_\X(\lambda)$
for any bounded Borel function $f$ on $\R$.

\subsection{Free Independence}
A family of von Neumann subalgebras $\left(\mathcal{A}_i\right)_{i\in I}$ 
of $\mathcal{A}$ 
is called \emph{free}
 if $\tau(\X_{1} \dots \X_{n} ) = 0$ whenever $\tau(\X_{j} ) = 0$ for all
$j = 1,\dots, n$ and $\X_{j} \in \mathcal{A}_{i(j)}$ for some indices $i(1)\neq i(2)\neq \dots \neq i(n)$.
Random variables $\X_{1},\dots ,\X_{n} $  are freely independent (free) if the
subalgebras they generate are free. 
Free random variables can be constructed using the reduced free product
of von Neumann algebras \cite{Voiculescu:1985}.

\subsection{Free Convolution and the Cauchy-Stieltjes Transform}
It can be shown that the joint distribution of free random variables $X_i$
is uniquely determined by the distributions of the individual random variables
$X_i$ and therefore the operation of \emph{free convolution} is well defined:
Let $\mu$ and $\nu$ be probability measures on $\R$, and
$\X,\Y$ self-adjoint free random variables with respective distributions
$\mu$ and $\nu$,
The distribution of $\X+\Y$ is called the free additive convolution of $\mu$
and $\nu$ and is denoted by $\mu \boxplus \nu$.
For more details about free convolutions and free probability theory we refer
the reader to the standard sources
\cite{VoiculescuDykemaNica:1992,NicaSpeicher:2006,MingoSpeicher:2017}.

The analytic approach to free convolution is based on the Cauchy transform
\begin{align}
\label{eq:hm:ball}
G_\mu(z)=\int_{\R}\frac{1}{z-y}\,\mu(dy)
\end{align}
of a probability measure $\mu$. The Cauchy transform is analytic on the upper half
plane $\IC^+=\{x+iy|x,y\in \R, y>0\}$ 
and takes values in the closed lower half plane
 $\IC^-\cup\IR$. For measures with compact support the Cauchy transform is analytic at infinity
and related to the moment generating function $M_{\X}$ as follows:
\begin{align}\label{mgf}
M_{\X}(z)=\sum_{n=0}^{\infty}\,\tau(\X^n)\,z^n = \frac{1}{z}\,G_\X (1/z).
\end{align}

Moreover the Cauchy transform has an inverse in some neighbourhood of infinity
which has the form
$$
G_\mu^{-1}(z) = \frac{1}{z} + R_\mu(z),
$$
where $R_\mu(z)$ is analytic in a neighbourhood of zero and is called 
\emph{$R$-transform}.
The coefficients of its series expansion
\begin{align} \label{rtr}
R_{\X}(z)=\sum_{n=0}^{\infty}\,K_{n+1}(\X)\,z^n
\end{align} 
are called \emph{free cumulants} of the random variable $\X$, see
Section~\ref{ssec:freecumulants} below. 
As a formal generating series
it will be convenient to consider  instead
the shift $\Rtrans_{X}(z):=zR_{X}(z)$
which is called 
\emph{free cumulant transform} or \emph{free cumulant generating function}, 
The free convolution can now be computed via the identity
\begin{align} \label{freeconv}
  R_{\mu\boxplus\nu}(z)=R_\mu(z)+R_\nu(z)
  ,
\end{align}
see~\cite{Voiculescu:1986}.

In order to accomodate for measures with noncompact support, the following
reformulation is useful \cite{BercoviciVoiculescu:1993:unbounded}.
Let $F_\mu(z)=1/G_\mu(z)$ be the reciprocal Cauchy transform.
Then $F_\mu(z)$ has an analytic right compositional inverse $F_\mu^{-1}$ on a region
$$
\Gamma_{\eta,M}  = \{ z\in \IC \mid \abs{\Re z} < \eta\Im z, \text{ } \Im z > M\}
;
$$
the \emph{Voiculescu transform} is defined as the function
$$
\phi_\mu(z) = F_\mu^{-1}(z) - z
$$
which turns out to be $\phi_\mu(z) = R_\mu(1/z)$.

\subsection{Free infinite divisibility}

In analogy with classical probability,
a probability measure $\mu$ on $\R$ is said to be
\emph{freely infinitely divisible} (or FID for short)
if for each $n \in \{1, 2, 3, \dots \}$ there exists a probability measure
$\mu_n$ such that
$\mu= \mu_n\boxplus\mu_n\boxplus\dots\boxplus\mu_n
$ ($n$-fold convolution).

Free infinite divisibility of a measure $\mu$ is characterized
by the property that its Voiculescu transform has
a Nevanlinna-Pick representation
\cite{BercoviciVoiculescu:1993:unbounded}
\begin{equation}\label{eq:LevyKhintchine}
  \phi_{\mu}(z)=
  \gamma+\int_{\R}\frac{1+xz}{z-x}\,\rho(dx)
  =
  \gamma+\int_{\R}\left(\frac{1}{z-x}+\frac{x}{1+x^2}\right)(1+x^2)\,\rho(dx)
\end{equation}
for some $\gamma \in \R$ and some nonnegative finite measure $\rho$.

Combinatorially, the characterization \eqref{eq:LevyKhintchine}
is equivalent to the statement that the sequence of free cumulants
is conditionally positive definite, i.e., for all $n\in\IN$ and all vectors $\xi\in\IC^n$
\begin{equation}
  \label{eq:condposdef}
\sum_{i,j=1}^n \xi_i\widebar{\xi}_j K_{i+j}(X)\geq 0.
\end{equation}
Equivalently, the Hankel determinants $[K_{i+j}(X)]_{i,j=1,2,\dots,n}$ are positive
for all $n\geq 1$,
see \cite[Lecture~13]{NicaSpeicher:2006}.

\subsection{Some probability distributions}
Let us now recall the basic properties of some specific probability distributions
which play prominent roles in the present paper.

\subsubsection{Wigner semicircle law}
The Wigner semicircle law  has density
\begin{equation}
  \label{eq:semicirclelaw}
  d\mu(x) = \frac{1}{2\pi}
  \sqrt{4-x^2}\,dx
\end{equation}
on $-2 \leq x \leq 2 $. Its 
Cauchy-Stieltjes transform
is given by the formula
\begin{equation}
G_{\mu}(z)=\frac{z  -\sqrt{z ^2 - 4}}{2} , \label{eq:GtransformataMixner}
\end{equation}
where  $|z|$ is big enough and where the branch of the analytic square root is determined by the condition
that $\Im(z)>0\Rightarrow \Im(G_\mu(z))\leqslant 0$ (see \cite{SaitohYoshida:2001}). 

A non-commutative random variable $\X$ with semicircle law
is called \emph{semicircular} or \emph{free gaussian} random variable.
The reason for the latter is the fact that its free cumulants $K_r=0$ for $r>2$
and it appears in the free version of the central limit theorem.

\subsubsection{Free Poisson law}
The \emph{Marchenko-Pastur distribution} or \emph{free Poisson distribution of rate $\lambda$}  has $R$-transform
$$
R(z) = \frac{\lambda}{1- z}
.
$$
Let $\rho$ be a probability measure on the real line.
The \emph{compound free Poisson distribution} with parameters $(\lambda,\rho)$ has $R$-transform
$$
R(z) = \lambda (M_\rho(z)-1),
$$
i.e., the free cumulants are $K_n=\lambda m_n(\rho)$.

\subsubsection{Tetilla law}
If 
$\X$
and
$\Y$
are two free semicircular random variables with  variance one
then the law $\mu$ of the commutator
$i(\X\Y-\Y\X)$ 
is supported on the interval $\abs{x}<\sqrt{{11+5\sqrt{5}}}$
and is absolutely continuous with density
$$
{\mu}(dx)=\frac{1}{2\sqrt{3}\pi \abs{x}}\left[\sqrt[3]{1+18{x}^2+3\sqrt{3x^2+33x^4-6x^6}}-\sqrt[3]{1+18{x}^2-3\sqrt{3x^2+33x^4-6x^6}}\right]\,dx.
$$ 
The above density is rescaled from \cite[equation  (2.8)]{DeyaNourdin:2012}.
The name \emph{tetilla law}
has its origin in the similarity of its density  with the tetilla cheese from Galicia,
see \cite{DeyaNourdin:2012}.
\subsubsection{Compound free Poisson
  distribution}
In particular the case when the free cumulants form a moment sequence,
i.e., $K_n(\mu)= \lambda m_n(\nu)$ for some $\lambda>0$ and some probability
measure  $\nu$. In this case $\mu$ is called a \emph{compound free Poisson
  distribution of rate $\lambda$
with jump distribution $\nu$}. 

\subsubsection{Even elements}
 We call an element $\X\in \mathcal{A}$ even if  all its odd moments vanish, i.e., $\tau(\X^{2i+1})=0$ for all $i\geq 0.$ 
It is immediate that the vanishing of all odd moments is
equivalent to the vanishing of all odd cumulants, i.e., $K_{2i+1}(\X)=0$  
and thus the even cumulants contain the complete information about the distribution of an even element. The sequence
$\alpha_n=K_{2n}(\X)$ of even cumulants is called the \emph{determining sequence} of $X$.

\subsection{Noncrossing Partitions}
We recall some facts about noncrossing partitions. For details and proofs see
the lecture notes \cite[Lecture~9]{NicaSpeicher:2006}.
Let $S\subseteq\N $ be a finite subset.
A partition of $S$ is a set of mutually disjoint subsets
(also called \emph{blocks}) $B_1,B_2,\dots,B_k\subseteq S$ 
whose union is $S$. The \emph{size} of $\pi$ is the number of blocks
and will be denoted by $\abs{\pi}$.
Any partition $\pi$ defines an equivalence relation on $S$,
denoted by $\sim_\pi$, such that the equivalence classes are the blocks $\pi$. 
That is, $i\sim_\pi j$ if $i$ and $j$ belong to the same block of $\pi$.
A partition $\pi$ is called \emph{noncrossing} 
if different blocks do not interlace, i.e., there is no quadruple of
elements $i<j<k<l$ such that $i\sim_\pi k$ and $j\sim_\pi l$ 
but $i\not\sim_\pi j$. 

The set of non-crossing partitions of $S$ is denoted by $\NC(S)$,
in the case where $S=[n]:=\{1, \dots , n\}$ we write
$\NC(n):=\NC([n])$. 
$\NC(n)$ is a lattice under \emph{refinement order}, where the relation $\pi\leq \rho$
holds if
every block of $\pi$ is contained in a block of $\rho$.
The subclass of noncrossing pair partitions (i.e., noncrossing complete matchings)
is denoted by $\NC_2(n)$.

The maximal element of $\NC(n)$ under this order is the partition consisting
of only one block and it is denoted by  $\hat{1}_{n}$.
On the other hand the minimal element $\hat{0}_n$ 
is the unique partition where every block is a singleton.
Sometimes it is convenient to visualize partitions as diagrams, for example
$\hat{1}_n=\makeatletter{}\begin{tikzpicture}[inner sep=0pt,scale=0.04]
\draw (2,0)--(2,7.5);
\draw (8,0)--(8,7.5);
\node at (18,3){$\cdots$};
\draw (26,0)--(26,7.5);
\draw (2,7.5)--(26,7.5);
\end{tikzpicture}
 $
and $\hat{0}_n=\makeatletter{}\begin{tikzpicture}[inner sep=0pt,scale=0.04]
\draw (2,0)--(2,4.5);
\draw (8,0)--(8,4.5);
\node at (18,2){$\cdots$};
\draw (26,0)--(26,4.5);
\draw (2,4.5)--(2,4.5);
\draw (8,4.5)--(8,4.5);
\draw (14,4.5)--(14,4.5);
\draw (20,4.5)--(20,4.5);
\draw (26,4.5)--(26,4.5);
\end{tikzpicture}
 $. 

We will be concerned with the parity of block sizes.
A block of a partition is called even (resp.~odd) according to the parity
of its cardinality.
A partition is called \emph{even} if each of its blocks has even cardinality.
For even $n\in\IN$  we denote by  $\NCeven(n)$ the subset of even noncrossing
partitions and we will call \emph{odd noncrossing partitions} the elements of the complement $\NCodd(n)$, i.e., those which have at least one odd block.

We will apply the product formula
\eqref{twr:produktargumentow} below only in the case
of pairwise products of random variables and in this case
two specific pair partitions and their complements will play a particularly important role,
namely the \emph{standard matching}
$\onetwo{n}=\makeatletter{}\begin{tikzpicture}[inner sep=0pt,scale=0.04]
\draw (2,0)--(2,4.5);
\draw (8,0)--(8,4.5);
\draw (14,0)--(14,4.5);
\draw (20,0)--(20,4.5);
\node at (30,2){$\cdots$};
\draw (38,0)--(38,4.5);
\draw (44,0)--(44,4.5);
\draw (2,4.5)--(8,4.5);
\draw (14,4.5)--(20,4.5);
\draw (38,4.5)--(44,4.5);
\end{tikzpicture}
 \in\NC(2n)$
and its shift $\pispecial{n}=\makeatletter{}\begin{tikzpicture}[inner sep=0pt,scale=0.04]
\draw (2,0)--(2,7.5);
\draw (8,0)--(8,4.5);
\draw (14,0)--(14,4.5);
\draw (20,0)--(20,4.5);
\draw (26,0)--(26,4.5);
\node at (36,2){$\cdots$};
\draw (44,0)--(44,4.5);
\draw (50,0)--(50,4.5);
\draw (56,0)--(56,7.5);
\draw (8,4.5)--(14,4.5);
\draw (20,4.5)--(26,4.5);
\draw (44,4.5)--(50,4.5);
\draw (2,7.5)--(56,7.5);
\end{tikzpicture}
 \in\NC(2n)$.
The blocks $(2k-1,2k)$, $k\in [n]$ of the standard matching $\onetwo{n}$, are called  \emph{standard braces}.

The action of the symmetric group $\SG_n$ on the set $[n]$ naturally
induces an action on set partitions, namely
\begin{equation}
  \label{eq:sigma.pi}
\sigma\cdot\pi=\{\sigma(B) \mid B\in\pi\}
.
\end{equation}

\subsection{Free Cumulants}
\label{ssec:freecumulants}
Given  a  noncommutative  probability space $(\A,\tau)$
 the free  cumulants are multilinear maps $K_n : \mathcal{A}^n\to\IC$ 
defined  implicitly in terms of the mixed moments by the relation
\begin{align}
\tau(\X_{1}\X_{2}\dots \X_{n}) = \sum_{\pi \in \NC(n)}K_{\pi}(\X_{1},\X_{2},\dots ,\X_{n}),\label{eq:DefinicjaKumulant}
\end{align}
where 
\begin{align}
K_{\pi}(\X_{1},\X_{2},\dots ,\X_{n}):=\Pi_{B \in \pi}K_{\abs{B}}(\X_{i}:i \in B). \label{eq:DefinicjaProduktuKumulant}
\end{align}
 Sometimes we will abbreviate univariate cumulants as $K_{n}(\X)=K_{n}(\X,\dots ,\X )$. 

The action  \eqref{eq:sigma.pi} of a permutation on noncrossing partitions may
introduce crossings. This is however not the case for cyclic permutations and
mirror permutations. 
We record their effect on tracial cumulants in the following lemma, which
follows directly from the corresponding properties of the trace.
\begin{lemm}
\label{lem:traciality}
  \label{lem:permutedcumulants}
  Let $X_1, X_2,\dots,X_n\in\mathcal{A}$ be random variables in a tracial
  probability space, then
  \begin{enumerate}[(i)]
   \item $K_n(X_1,X_2,\dots,X_n)= \overline{K_n(X_n^*,X_{n-1}^*,\dots,X_1^*)}$
   \item $K_n(X_2,X_3,\dots,X_n,X_1)=K_n(X_1,X_2,\dots,X_n).$
  \end{enumerate}

\end{lemm}

Free cumulants provide a powerful technical tool to investigate
free random variables. 
This is due to the basic property of \emph{vanishing of mixed cumulants}. 
By this we mean the property that 
$$
K_n(X_1,X_2,\dots,X_n)=0
$$
for any family of random variables $X_1,X_2,\dots,X_n$ which can be partitioned
into two mutually free subsets.

For free sequences this can be reformulated as follows.
Let $(X_i)_{i\in\N}$ be a sequence of free random variables and
$h:[r]\to\N$ a map. We denote by $\ker h$ the set partition
which is induced by the equivalence relation 
$$
i\sim_{\ker h} j
\ 
\iff
\ 
h(i)=h(j)
.
$$
In this notation, vanishing of mixed cumulants implies that
\begin{equation}
  \label{eq:kerh>=pi}
  K_\pi(X_{h(1)},X_{h(2)},\dots,X_{h(r)})=0
  \text{ unless $\ker h\geq \pi$.}
\end{equation}

Our main technical tool is the free version,
due to Krawczyk and Speicher~\cite{KrawczykSpeicher:2000} (see also   \cite[Theorem 11.12]{NicaSpeicher:2006}), of the classical formula of 
James and Leonov/Shiryaev \cite{James:1958,LeonovShiryaev:1959} which
expresses cumulants of products in terms of individual cumulants.
\begin{theo}
\label{thm:krawczyk}
Let
$r,n \in \N$ and $ i_1 < i_2 < \dots < i_r = n$ be given and let
$$\rho=\{(1,2,\dots,i_1),(i_1+1,i_1+2,\dots,i_2),\dots,(i_{r-1}+1,i_{r-1}+2,\dots,i_r)\}\in \NC(n)$$ 
be the induced interval partition.
Consider now random variables $\X_1,\dots,\X_n\in\A$.
Then the free cumulants of the products can be expanded
as follows:
\begin{align} 
\label{twr:produktargumentow}
K_r(\X_1\dots \X_{i_1},\dots,\X_{i_{r-1}+1}\dots\X_n)=\sum_{\substack{\pi\in \NC(n) \\ \pi\vee\rho=\hat{1}_{n}} }K_\pi({\X_1,\dots,\X_n}). 
\end{align} 
\end{theo}

\subsection{Cumulants of quadratic forms}
Kreweras \cite{Kreweras:1972} discovered an interesting antiisomorphism
of the lattice $\NC(n)$, now called the \emph{Kreweras complementation map},
of which we will need two variants.
Given a noncrossing partition $\pi$ of $\{1,2,\dots,n\}$,
the \emph{left Kreweras complement} $\Krewl\pi$ is the maximal
noncrossing partition of the ordered set $\{\bar{1},\bar{2},\dots,\bar{n}\}$
such that $\pi\cup\Krewl\pi$ is a noncrossing partition of
the interlaced set
$\{\bar{1},1,\bar{2},2,\dots,\bar{n},n\}$.
Similarly, the \emph{right Kreweras complement} $\Krewr\pi$ is the maximal
noncrossing partition of the ordered set $\{\bar{1},\bar{2},\dots,\bar{n}\}$
such that $\pi\cup\Krewr\pi$ is a noncrossing partition of
the interlaced set
$\{1,\bar{1},2,\bar{2},\dots,n,\bar{n}\}$.
The two maps are inverse to each other and it can be shown that the sizes
are related by the identity
\begin{equation}
  \label{eq:cardKrew}
  \abs{\Krewr{\pi}}=\abs{\Krewl\pi}=n+1-\abs\pi
  .
\end{equation}
This motivates the following definition.
\begin{defi}[{\cite[Ch.~17]{NicaSpeicher:2006}}]
  \label{def:boxedconvolution}
  Let $$f(z_1,\dots,z_m)=\sum_{n=1}^\infty \sum_{i_1,\dots,i_n=1}^m  a_{z_{i_1},\dots,z_{i_n}} z_{i_1}\dots z_{i_n} \text{   
   and }g(z_1,\dots,z_m)=\sum_{n=1}^\infty \sum_{i_1,\dots,i_n=1}^m  b_{z_{i_1},\dots,z_{i_n}} z_{i_1}\dots z_{i_n},$$
  be two formal noncommutative power series. Their \emph{boxed convolution} is
  defined as the coefficient of order $(i_1,\dots,i_n)$ of
  the formal power series
  $f\boxstar g$ which is defined as   $$
 \Cf_{(i_1,\dots,i_n)}(f\boxstar g) = \sum_{\pi\in\NC(n)} \Cf_{(i_1,\dots,i_n),\pi}(f) \Cf_{(i_1,\dots,i_n),\Krewr{\pi}}(g)
        .
  $$
\end{defi}
The boxed convolution $\boxstar$ ist most frequently used with the so
called the \emph{Zeta-series} and the  \emph{M\"obius-series}, which are
defined as
\begin{align*}
 \zeta_m(z_1,\dots,z_m)&=\sum_{n=1}^\infty \sum_{i_1,\dots,i_n=1}^m
                         z_{i_1}\dots z_{i_n}
  \\
 \Moeb_m(z_1,\dots,z_m)&=\sum_{n=1}^\infty(-1)^{n+1}\frac{(2n-2)!}{(n-1)!n!}(z_1+\dots+z_m)^n  
.
\end{align*}
The functions 
$\zeta_m$ and $\Moeb_m$ are inverse to each other
with respect to $\boxstar$.
In order to compute cumulants of quadratic forms we use the following results
from our previous paper \cite{EjsmontLehner:2017}.
\begin{lemm} [{\cite[Lemma~2.14]{EjsmontLehner:2017}}]
  \label{lemm:lematoparzystych}

  Let $r\in\mathbb{N}$ and $\pi \in \NCeven(2r)$,
  then $\pi \vee \onetwo{r}=\hat{1}_{2r}$ if and only if
  $\pi\geq \pispecial{r}=\makeatletter{}\begin{tikzpicture}[inner sep=0pt,scale=0.04]
\draw (2,0)--(2,7.5);
\draw (8,0)--(8,4.5);
\draw (14,0)--(14,4.5);
\draw (20,0)--(20,4.5);
\draw (26,0)--(26,4.5);
\node at (36,2){$\cdots$};
\draw (44,0)--(44,4.5);
\draw (50,0)--(50,4.5);
\draw (56,0)--(56,7.5);
\draw (8,4.5)--(14,4.5);
\draw (20,4.5)--(26,4.5);
\draw (44,4.5)--(50,4.5);
\draw (2,7.5)--(56,7.5);
\end{tikzpicture}
 $,
  i.e., $1$ and $2r$ lie in the same block of $\pi$ and elements $2i$ and $2i+1$
  also lie in the same block of $\pi$ for $i\in[r-1]$.  
  Consequently, 
  $$
  \{\pi : \pi \vee \onetwo{r}=\hat{1}_{2r}\}\cap \NCeven(2r)
  = [\pispecial{r}, \hat{1}_{2r}  ] ,
  $$
  is a lattice isomorphic to $\NC(r)$.

\end{lemm}

We will use the following result from \cite{EjsmontLehner:2017} to
express cumulants of quadratic forms in even random variables 
in terms of the diagonal map of matrices.
This is the conditional expectation $\ED$ which which annihilates
all off-diagonal entries of a matrix, i.e.,
if we denote by $E_i$  the projection matrix onto the subspace
spanned by the  $i$-th unit vector, then
  \begin{equation}
  \label{eq:ED}
\begin{aligned}
  \ED:M_n(\IC)&\to M_n(\IC)\\
  A &\mapsto \sum_{i=1}^n E_iAE_i.
\end{aligned}
\end{equation}

\begin{prop}[{\cite[Proposition~4.5]{EjsmontLehner:2017}}]
  \label{prop:CykliczneVariancja}
  Let $\X_1, \X_2,\dots, \X_n\in \A$ be a free family of even random variables,
  $\bX= [X_iX_j]_{i,j=1  }^n$, $A=[a_{i,j}]_{i,j=1}^n\in M_n(\C)$ a scalar
  matrix
  and $Q_n=\sum a_{i,j}X_iX_j$ a quadratic form.
  \begin{enumerate}[(i)]
   \item \label{it:cyclic3}
    The cumulants of $Q_n$ are given by
\begin{equation}  \label{eq:kumulantsamplevariancenotiid}
 K_r(Q_n) =\sum_{i_1,\dots,i_r\in[n]} 
            \Tr(AE_{i_1}AE_{i_2}\dots AE_{i_r})\,
            \sum_{\substack{ \pi\in \NCeven(2r)\\ 
                \pi \vee \onetwo{r}=\hat{1}_{2r}}}             K_\pi(X_{i_r},X_{i_1},X_{i_1},X_{i_2},\dots,X_{i_{r-1}},X_{i_r}).
    \end{equation}
    \item If we assume in addition that $X_i$ are identically distributed
     then
     the previous formula simplifies to the following convolution-like expression
    \begin{equation}  \label{eq:kumulantsamplevariance}
      K_r(Q_n)=\sum_{ \pi\in \NC(r)}
      \Tr(\ED[\Krewl{\pi}](A)) \prod_{B\in\pi}K_{2\abs{B}}(X).
    \end{equation} 
  \end{enumerate}
\end{prop}

\begin{Rem}
  In the case of a free standard semicircular family formula 
  \eqref{eq:kumulantsamplevariance} has only one contributing term and
  takes the particularly simple form
  \begin{equation}
    \label{eq:cumQnsemi}
    K_r(Q_n) = \Tr(A_n^r)
    .
  \end{equation}
\end{Rem}

\subsection{Special notations and definitions for noncrossing partitions}
\begin{defi}
  A lattice $(L,\leq)$ is called \emph{bounded} if it has a unique minimal 
  and a unique maximal element,
  usually denoted $\hat{0}$ and $\hat{1}$, respectively.
  Let $a\in L$. An element $b\in L$ is called a \emph{complement} of $a$
  if  $a\wedge b=\hat{0}$ 
  and $a\vee b=\hat{1}$.
  We will need the weaker notion of \emph{upper complements}, i.e.,
  the set
  $$
    \{b\in L \mid a\vee b=\hat{1}\}
  .
  $$
\end{defi}

We denote the set of upper complements of $\onetwo{n}$ in $\NC(2n)$ by
$$
\Comp{2n} = \{\pi\in\NC(2n) \mid \pi\vee\makeatletter{}\begin{tikzpicture}[inner sep=0pt,scale=0.04]
\draw (2,0)--(2,4.5);
\draw (8,0)--(8,4.5);
\draw (14,0)--(14,4.5);
\draw (20,0)--(20,4.5);
\node at (30,2){$\cdots$};
\draw (38,0)--(38,4.5);
\draw (44,0)--(44,4.5);
\draw (2,4.5)--(8,4.5);
\draw (14,4.5)--(20,4.5);
\draw (38,4.5)--(44,4.5);
\end{tikzpicture}
 =\hat{1}_{2n}
                                      \}.
$$
Among these we single out the even ones
$$
\CompEven{2n} = \Comp{2n}\cap \NCeven(2n),
$$
and the remaining ones
$$
\CompOdd{2n} = \Comp{2n}\setminus \NCeven(2n),
$$
which have at least one odd block. Our aim is to show that under certain
conditions the contributions of $\CompOdd{2n}$ in the expansion
\eqref{twr:produktargumentow} cancel each other. To this end we will define an
involution on $\CompOdd{2n}$
in Section~\ref{sec:involution}.
This involution is based on the concept of \emph{inner odd blocks},
which we present next.

\begin{defi}
Let $\pi$ be a noncrossing partition and $B,B'\in\pi$ two distinct blocks of $\pi$.
\begin{enumerate}[1.]
 \item 
  We denote by $\alpha(B)=\min B$ and $\omega(B)=\max B$ its extreme points.
  The interval $I(B)=[\alpha(B),\alpha(B)+1,\dots,\omega(B)]$ is called the 
  \emph{padding interval} of $B$. 
  
 \item Given another block  $B'\in\pi$ we say that $B'$ is \emph{nested inside} $B$
  if $I(B')\subseteq I(B)$, i.e., if $\alpha(B)<\alpha(B')\leq \omega(B')< \omega(B)$.
 \item 
  An \emph{inner odd block} of $\pi$ is a block $B\in\pi$ 
  such that no other odd block of $\pi$ is nested inside $B$.
  In particular, every singleton is an inner odd block.
  Let us emphasize that for our purposes
  we allow even blocks to be nested inside inner odd blocks;
  see Figure~\ref{fig:innerodd}
  for  examples.
  
\begin{figure}[ht]
  \centering
\makeatletter{} \begin{tikzpicture}[scale=0.05]
\draw[thline] (2,0)--(2,7.5);
\draw[thline] (20,0)--(20,7.5);
\draw[thline] (26,0)--(26,7.5);
\draw[thline] (2,7.5)--(26,7.5);

\draw (8,0)--(8,4.5);
\draw (14,0)--(14,4.5);

\draw (32,0)--(32,7.5);
\draw (56,0)--(56,7.5);
\draw (62,0)--(62,7.5);
\draw (32,7.5)--(62,7.5);

\draw[thline] (38,0)--(38,4.5);
\draw[thline] (44,0)--(44,4.5);
\draw[thline] (50,0)--(50,4.5);

\draw[thline] (68,0)--(68,7.5);
\draw[thline] (86,0)--(86,7.5);
\draw[thline] (92,0)--(92,7.5);
\draw[thline] (110,0)--(110,7.5);
\draw[thline] (116,0)--(116,7.5);
\draw[thline] (134,0)--(134,7.5);
\draw[thline] (140,0)--(140,7.5);
\draw[thline] (68,7.5)--(140,7.5);

\draw (74,0)--(74,4.5);
\draw (80,0)--(80,4.5);

\draw (98,0)--(98,4.5);
\draw (104,0)--(104,4.5);

\draw (122,0)--(122,4.5);
\draw (128,0)--(128,4.5);

\draw (146,0)--(146,4.5);
\draw (152,0)--(152,4.5);

\draw[thline] (158,0)--(158,4.5);
\draw[thline] (164,0)--(164,4.5);
\draw[thline] (170,0)--(170,4.5);
\draw[thline] (158,4.5)--(170,4.5);

\draw (8,4.5)--(14,4.5);
\draw (38,4.5)--(50,4.5);
\draw (74,4.5)--(80,4.5);
\draw (98,4.5)--(104,4.5);
\draw (122,4.5)--(128,4.5);
\draw (146,4.5)--(152,4.5);
\end{tikzpicture} 

  \caption{Some inner odd blocks}
  \label{fig:innerodd}
\end{figure}
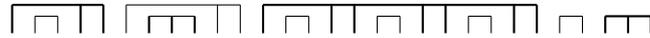
\end{enumerate}
\end{defi}

Before proceeding with further definitions we record in the next 
lemma some preliminary facts about inner odd blocks.

\begin{lemm} \label{lem:inneroddbasic} 
  Let $\pi\in\NCodd(2n)$, then   \begin{enumerate}[(i)]
   \item $\pi$ has at least one inner odd block.
   \item If $B\in\pi$ is an inner odd block,
    then its padding interval $I(B)$ is odd, the end points $\alpha(B)$ and $\omega(B)$ have
    the same parity and the complement $[2n]\setminus B$ 
    is a union of intervals, out of which exactly one is odd.
  \end{enumerate}
\end{lemm}

\begin{proof}
  \begin{enumerate}[(i)]
   \item
    The first part is obvious if $\pi$ has a singleton, otherwise 
    pick     any odd block $B\in \pi$.
    If it does not cover another odd block, we are done.
    Otherwise we choose      any odd block nested inside $B$ and
    continue the procedure recursively until an inner odd block is reached.
   \item  Let $I(B)$ be an interval. Then the complement of $I(B)$ is a union 
    of intervals and by definition those intervals which are covered 
    by $B$ are even.
    It follows that the padding interval $I(B)$ is the union of even blocks and exactly one odd
    block and therefore odd.
    Consequently $\alpha(B)$ and
    $\omega(B)$ have the same parity and exactly one of the ``outer'' intervals $[1,\alpha(B)-1]$
    and $[\omega(B)+1,2n]$ is odd (if $I(B)=[1,2,..,2k+1]$, then $[1,\alpha(B)-1]$ is empty set). 
  \end{enumerate}
\end{proof}

\begin{defi}
  Let $\pi$ be a noncrossing partition and $B,B'\in\pi$ two distinct blocks of $\pi$.

  \begin{enumerate}[1.]
\item Given a block $C\in \rho$ from another partition $\rho$, we say
  that $C$ \emph{connects} $B$ and $B'$ if both $B\cap C\ne\emptyset$ and $B'\cap C\ne\emptyset$.
\item Let $\pi\in \CompOdd{2n}$ and $B$ its  leftmost  inner odd block.
By Lemma~\ref{lem:inneroddbasic} the padding interval $I(B)$ has odd length
and therefore there is a unique standard brace $\SpPair_\pi\in  \onetwo{n}$
such that  $I(B)\cap \SpPair_\pi$ contains exactly one element.
We call $\SpPair_\pi$ the \emph{pivot brace} of $\pi$. 
The unique point in the intersection of the pivot brace and the leftmost
inner odd block is called the \emph{pivot element}.
In our figures the pivot brace will be highlighted by thick lines.
Associated to the pivot brace we call the two unique
  blocks $\BL,\BR\in \pi$ such that $\BL\cap \SpPair_\pi\neq \emptyset$ and $\BR\cap \SpPair_\pi\neq \emptyset$ the \emph{left} and \emph{right} \emph{pivot blocks} of $\pi$.
 \item For the pivot blocks we define the \emph{essentially nested blocks}, 
namely if $A\in\{\BL,\BR\}$, then 
$$
\nest(A):=\{B\mid B\in\pi \text{ and }  B \text{ nest inside } A\setminus \SpPair_\pi\},
$$
 where the notation $A\setminus \SpPair_\pi$ means that we remove those
 elements of $A$ which are included in the pivot block $\SpPair_\pi$,
see Figure~\ref{fig:nest}.
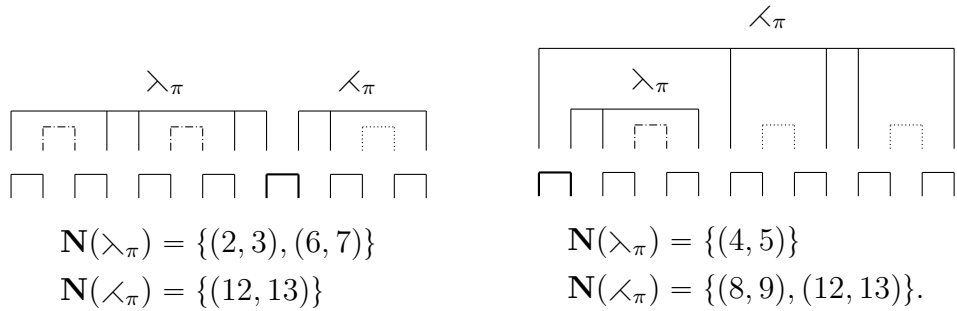
\begin{figure}[ht]
   \begin{subfigure}[b!]{0.4\textwidth}
  \centering
    \makeatletter{}\begin{tikzpicture}[scale=0.07,baseline,remember picture]
\node[color=white] at (31,27.675) {.};
\node[color=white] at (31,-5) {.};

\draw (2,0)--(2,7.5);
\draw[style=aline] (8,0)--(8,4.5);
\draw[style=aline] (14,0)--(14,4.5);
\draw[style=aline] (8,4.5)--(14,4.5);
\draw (20,0)--(20,7.5);
\draw (26,0)--(26,7.5);
\draw[style=aline] (32,0)--(32,4.5);
\draw[style=aline] (38,0)--(38,4.5);
\draw[style=aline] (32,4.5)--(38,4.5);
\draw (44,0)--(44,7.5);
\draw (50,0)--(50,7.5);
\draw (56,0)--(56,7.5);
\draw (62,0)--(62,7.5);
\draw[style=bline] (68,0)--(68,4.5);
\draw[style=bline] (74,0)--(74,4.5);
\draw[style=bline] (68,4.5)--(74,4.5);
\draw (80,0)--(80,7.5);
\draw (2,7.5)--(50,7.5);
\draw (56,7.5)--(80,7.5);

\node at (31,13) {$\BL$};
\node at (67,13) {$\BR$};

\draw (2,-9)--(2,-4.5);
\draw (8,-9)--(8,-4.5);
\draw (14,-9)--(14,-4.5);
\draw (20,-9)--(20,-4.5);
\draw (26,-9)--(26,-4.5);
\draw (32,-9)--(32,-4.5);
\draw (38,-9)--(38,-4.5);
\draw (44,-9)--(44,-4.5);
\draw[thick] (50,-9)--(50,-4.5);
\draw[thick] (56,-9)--(56,-4.5);
\draw[thick] (50,-4.5)--(56,-4.5);
\draw (62,-9)--(62,-4.5);
\draw (68,-9)--(68,-4.5);
\draw (74,-9)--(74,-4.5);
\draw (80,-9)--(80,-4.5);
\draw (2,-4.5)--(8,-4.5);
\draw (14,-4.5)--(20,-4.5);
\draw (26,-4.5)--(32,-4.5);
\draw (38,-4.5)--(44,-4.5);
\draw (62,-4.5)--(68,-4.5);
\draw (74,-4.5)--(80,-4.5);
\end{tikzpicture}

\begin{align*}
\nest(\BL)&=\{{   (2,3),(6,7)}\} \\
\nest(\BR)&=\{{(12,13)}\}
\end{align*}
  \end{subfigure}
~
  \begin{subfigure}[b!]{0.4\textwidth}
 \centering
\makeatletter{}\begin{tikzpicture}[scale=0.07,baseline,remember picture]
\node[color=white] at (31,25) {.};
\node[color=white] at (31,-5) {.};

\draw (2,0)--(2,19);
\draw (8,0)--(8,7.5);
\draw (14,0)--(14,7.5);
\draw[style=aline] (20,0)--(20,4.5);
\draw[style=aline] (26,0)--(26,4.5);
\draw (32,0)--(32,7.5);
\draw (38,0)--(38,19);
\draw[style=bline] (44,0)--(44,4.5);
\draw[style=bline] (50,0)--(50,4.5);
\draw (56,0)--(56,19);
\draw (62,0)--(62,19);
\draw[style=bline] (68,0)--(68,4.5);
\draw[style=bline] (74,0)--(74,4.5);
\draw (80,0)--(80,19);
\draw[style=aline] (20,4.5)--(26,4.5);
\draw (8,7.5)--(32,7.5);
\draw[style=bline] (44,4.5)--(50,4.5);
\draw[style=bline] (68,4.5)--(74,4.5);
\draw (2,19)--(80,19);

\node at (23,13) {$\BL$};
\node at (45,25) {$\BR$};

\draw[thick] (2,-9)--(2,-4.5);
\draw[thick] (8,-9)--(8,-4.5);
\draw (14,-9)--(14,-4.5);
\draw (20,-9)--(20,-4.5);
\draw (26,-9)--(26,-4.5);
\draw (32,-9)--(32,-4.5);
\draw (38,-9)--(38,-4.5);
\draw (44,-9)--(44,-4.5);
\draw(50,-9)--(50,-4.5);
\draw(56,-9)--(56,-4.5);
\draw(50,-4.5)--(56,-4.5);
\draw (62,-9)--(62,-4.5);
\draw (68,-9)--(68,-4.5);
\draw (74,-9)--(74,-4.5);
\draw (80,-9)--(80,-4.5);
\draw[thick] (2,-4.5)--(8,-4.5);
\draw (14,-4.5)--(20,-4.5);
\draw (26,-4.5)--(32,-4.5);
\draw (38,-4.5)--(44,-4.5);
\draw (62,-4.5)--(68,-4.5);
\draw (74,-4.5)--(80,-4.5);
\end{tikzpicture}

\begin{align*}
\nest(\BL)&=\{{               (4,5)}\} \\
 \nest(\BR)&=\{{             (8,9),(12,13)}\}.
\end{align*}
  \end{subfigure}

  \caption{Examples of pivot blocks $\BL$, $\BR$, pivot braces braces $\SpPair_\pi$ and  essentially nested blocks.} \label{fig:nest}
\end{figure}
\end{enumerate}
\end{defi}

For scalar $a,b, c\in\C$  we denote by  $\left[\begin{smallmatrix}
    c &a\\
    b& c
    \end{smallmatrix}\right]_n$
the element in $M_n(\C)$, where the diagonal elements are equal to $c$ and the  upper-triangular entries are equal to a and lower-triangular elements are $b$.

\section{An involution on $\CompOdd{2n}$}
\label{sec:involution}

We illustrate the idea of the proof of the cancellation phenomenon
on the simplest example which is the  commutator $XY-YX$.
We expand the cumulant $K_r(XY-YX)$ multilinearly, apply the product
formula \eqref{twr:produktargumentow} and obtain a sum
\begin{equation}
  \label{eq:commutatorexpansion}
\sum_{\substack{\pi\in \NC(2r) \\ \pi\vee\makeatletter{}\begin{tikzpicture}[inner sep=0pt,scale=0.04]
\draw (2,0)--(2,4.5);
\draw (8,0)--(8,4.5);
\draw (14,0)--(14,4.5);
\draw (20,0)--(20,4.5);
\node at (30,2){$\cdots$};
\draw (38,0)--(38,4.5);
\draw (44,0)--(44,4.5);
\draw (2,4.5)--(8,4.5);
\draw (14,4.5)--(20,4.5);
\draw (38,4.5)--(44,4.5);
\end{tikzpicture}
 =\hat{1}_{2r}} }  K_\pi(X_1,X_2,\dots,X_n)
\end{equation}
where for $X_1,X_2,\dots,X_n\in \{X,Y\}$;
our involution will then provide a matching of equal terms with opposite signs
roughly by shifting the endpoint of leftmost inner odd block according
to  the  pattern shown in Figure~\ref{fig:involutiontypeIII}.

However the definition of the involution is not as straightforward as it seems
at a first glance. The complication arises from the fact that
for certain partitions the leftmost inner odd block loses its property
of being leftmost after the shift, see example \eqref{eq:typeIIbad} below.
For this reason these partitions must be treated differently,
by ``flipping'' certain intervals.
Therefore we will call them \emph{flip partitions}
and their description is the content of the next subsection.

\subsection{Flip partitions} 

  \begin{lemm} \label{lemtwoinnerodd}
    Let $\pi\in\CompOdd{2n}$.     If $\pi$ has  two inner odd blocks which are connected by a standard brace
    then all other blocks of $\pi$ are even.
  \end{lemm}

  \begin{proof}
    If two inner odd blocks $B$ and $B'$ are connected, they must lie adjacent to each other.
    Let  $B$ be  to the left of $B'$,
    then $\omega(B)$ is odd and $\alpha(B')$
    is even, see Figure~\ref{fig:adjacentinnerodd}. But then the interval $J=I(B)\cup I(B')$ is even and the blocks
    of $\onetwo{n}$ which are contained in $J$ are not connected to those
    contained in the complement of $J$.
    Since $\pi\in \CompOdd{2n}$, it follows that $J=[2n]$ and since 
    no odd block is nested neither inside  $B$ nor inside $B'$,
    $\pi$ has no other odd blocks.
  \end{proof}

  \begin{figure}[h]
    \centering
    \makeatletter{}\begin{tikzpicture}[scale=0.07]
\draw[thline] (2,0)--(2,7.5);
\draw[thline] (20,0)--(20,7.5);
\draw[thline] (26,0)--(26,7.5);
\draw[thline] (44,0)--(44,7.5);
\draw[thline] (50,0)--(50,7.5);
\draw[thline] (2,7.5)--(50,7.5);

\draw (8,0)--(8,4.5);
\draw (14,0)--(14,4.5);

\draw (32,0)--(32,4.5);
\draw (38,0)--(38,4.5);

\draw[thline] (56,0)--(56,7.5);
\draw[thline] (62,0)--(62,7.5);
\draw[thline] (80,0)--(80,7.5);
\draw[thline] (56,7.5)--(80,7.5);

\draw (68,0)--(68,4.5);
\draw (74,0)--(74,4.5);
\draw (8,4.5)--(14,4.5);
\draw (32,4.5)--(38,4.5);
\draw (68,4.5)--(74,4.5);

\draw (2,-9)--(2,-4.5);
\draw (8,-9)--(8,-4.5);
\draw (14,-9)--(14,-4.5);
\draw (20,-9)--(20,-4.5);
\draw (26,-9)--(26,-4.5);
\draw (32,-9)--(32,-4.5);
\draw (38,-9)--(38,-4.5);
\draw (44,-9)--(44,-4.5);
\draw[thline] (50,-9)--(50,-4.5);
\draw[thline] (56,-9)--(56,-4.5);
\draw (62,-9)--(62,-4.5);
\draw (68,-9)--(68,-4.5);
\draw (74,-9)--(74,-4.5);
\draw (80,-9)--(80,-4.5);
\draw (2,-4.5)--(8,-4.5);
\draw (14,-4.5)--(20,-4.5);
\draw (26,-4.5)--(32,-4.5);
\draw (38,-4.5)--(44,-4.5);
\draw[thline] (50,-4.5)--(56,-4.5);
\draw (62,-4.5)--(68,-4.5);
\draw (74,-4.5)--(80,-4.5);
\end{tikzpicture}
 
    \caption{Two adjacent inner odd blocks.}
    \label{fig:adjacentinnerodd}
  \end{figure}
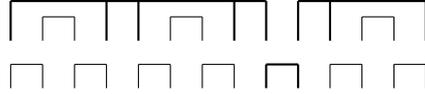

\begin{defi} An element  $\pi\in \CompOdd{2n}$ is called a
  \textit{flip partition} if it has exactly two odd blocks and the pivot brace
  connects these two blocks at their endpoints. \label{def:flippartition}
\end{defi}
Examples of flip partitions are shown in Figure~\ref{fig:flippartitions}.

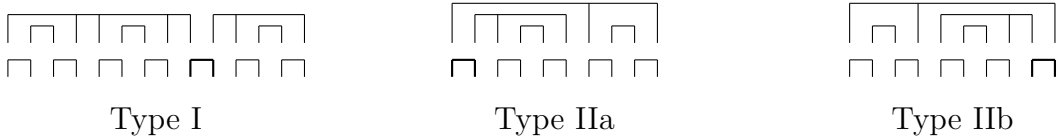
\begin{figure}[h]
  \begin{subfigure}[t]{0.3\textwidth}
 \centering
 \makeatletter{}\begin{tikzpicture}[scale=0.05]
\draw (2,0)--(2,7.5);
\draw (8,0)--(8,4.5);
\draw (14,0)--(14,4.5);
\draw (20,0)--(20,7.5);
\draw (26,0)--(26,7.5);
\draw (32,0)--(32,4.5);
\draw (38,0)--(38,4.5);
\draw (44,0)--(44,7.5);
\draw (50,0)--(50,7.5);
\draw (56,0)--(56,7.5);
\draw (62,0)--(62,7.5);
\draw (68,0)--(68,4.5);
\draw (74,0)--(74,4.5);
\draw (80,0)--(80,7.5);
\draw (8,4.5)--(14,4.5);
\draw (32,4.5)--(38,4.5);
\draw (2,7.5)--(50,7.5);
\draw (68,4.5)--(74,4.5);
\draw (56,7.5)--(80,7.5);

\draw (2,-9)--(2,-4.5);
\draw (8,-9)--(8,-4.5);
\draw (14,-9)--(14,-4.5);
\draw (20,-9)--(20,-4.5);
\draw (26,-9)--(26,-4.5);
\draw (32,-9)--(32,-4.5);
\draw (38,-9)--(38,-4.5);
\draw (44,-9)--(44,-4.5);
\draw[thline] (50,-9)--(50,-4.5);
\draw[thline] (56,-9)--(56,-4.5);
\draw (62,-9)--(62,-4.5);
\draw (68,-9)--(68,-4.5);
\draw (74,-9)--(74,-4.5);
\draw (80,-9)--(80,-4.5);
\draw (2,-4.5)--(8,-4.5);
\draw (14,-4.5)--(20,-4.5);
\draw (26,-4.5)--(32,-4.5);
\draw (38,-4.5)--(44,-4.5);
\draw[thline] (50,-4.5)--(56,-4.5);
\draw (62,-4.5)--(68,-4.5);
\draw (74,-4.5)--(80,-4.5);
\end{tikzpicture} 
\vskip1ex

  Type I
 \end{subfigure}
~
  \begin{subfigure}[t]{0.3\textwidth}
 \centering
 \makeatletter{}
\begin{tikzpicture}[scale=0.05]
\draw (2,0)--(2,10.5);
\draw (8,0)--(8,7.5);
\draw (14,0)--(14,7.5);
\draw (20,0)--(20,4.5);
\draw (26,0)--(26,4.5);
\draw (32,0)--(32,7.5);
\draw (38,0)--(38,10.5);
\draw (44,0)--(44,4.5);
\draw (50,0)--(50,4.5);
\draw (56,0)--(56,10.5);
\draw (20,4.5)--(26,4.5);
\draw (8,7.5)--(32,7.5);
\draw (44,4.5)--(50,4.5);
\draw (2,10.5)--(56,10.5);

\draw[thline] (2,-9)--(2,-4.5);
\draw[thline] (8,-9)--(8,-4.5);
\draw (14,-9)--(14,-4.5);
\draw (20,-9)--(20,-4.5);
\draw (26,-9)--(26,-4.5);
\draw (32,-9)--(32,-4.5);
\draw (38,-9)--(38,-4.5);
\draw (44,-9)--(44,-4.5);
\draw (50,-9)--(50,-4.5);
\draw (56,-9)--(56,-4.5);
\draw[thline] (2,-4.5)--(8,-4.5);
\draw (14,-4.5)--(20,-4.5);
\draw (26,-4.5)--(32,-4.5);
\draw (38,-4.5)--(44,-4.5);
\draw (50,-4.5)--(56,-4.5);
\end{tikzpicture}

\vskip1ex

  Type IIa
 \end{subfigure}
~
  \begin{subfigure}[t]{0.3\textwidth}
 \centering
\makeatletter{}\begin{tikzpicture}[scale = 0.05]
\draw (2,0)--(2,10.5);
\draw (8,0)--(8,4.5);
\draw (14,0)--(14,4.5);
\draw (20,0)--(20,10.5);
\draw (26,0)--(26,7.5);
\draw (32,0)--(32,4.5);
\draw (38,0)--(38,4.5);
\draw (44,0)--(44,7.5);
\draw (50,0)--(50,7.5);
\draw (56,0)--(56,10.5);
\draw (8,4.5)--(14,4.5);
\draw (32,4.5)--(38,4.5);
\draw (26,7.5)--(50,7.5);
\draw (2,10.5)--(56,10.5);

\draw (2,-9)--(2,-4.5);
\draw (8,-9)--(8,-4.5);
\draw (14,-9)--(14,-4.5);
\draw (20,-9)--(20,-4.5);
\draw (26,-9)--(26,-4.5);
\draw (32,-9)--(32,-4.5);
\draw (38,-9)--(38,-4.5);
\draw (44,-9)--(44,-4.5);
\draw[thline] (50,-9)--(50,-4.5);
\draw[thline] (56,-9)--(56,-4.5);
\draw (2,-4.5)--(8,-4.5);
\draw (14,-4.5)--(20,-4.5);
\draw (26,-4.5)--(32,-4.5);
\draw (38,-4.5)--(44,-4.5);
\draw[thline] (50,-4.5)--(56,-4.5);
\end{tikzpicture} 

\vskip1ex

  Type IIb
 \end{subfigure}
  \caption{Examples of flip partitions}
  \label{fig:flippartitions}
\end{figure}
The next lemma provides us with a classification of flip partitions
which will be essential for the definition of the involution.

\begin{lemm} \label{lemma:symetryczne}
Let   $\pi\in \CompOdd{2n}$ be a  flip  partition. Then either
 \begin{enumerate}[I.]
\item \label{it:symetryczne:1} $\pi$  has two  inner odd blocks; 
\item \label{it:symetryczne:2}  $\pi$  has exactly one inner odd block and the pivot brace is  either 
\begin{enumerate}[(a)]
 \item $\SpPair_\pi=(1,2)$
 \item $\SpPair_\pi=(2n-1,2n)$.
\end{enumerate}
\end{enumerate}    
\end{lemm}
We refer to flip partitions of type I, IIa and IIb according to this scheme,
see Figure~\ref{fig:flippartitions}.
\begin{proof}
Suppose that $\pi$ is a flip partition and condition \eqref{it:symetryczne:1} 
is not satisfied. Then $\pi$ has one inner odd block $B$ and one 
outer odd block $B'$.  The padding interval $I(B')$ of the outer
odd block is the union of $B\cup B'$ and some even blocks
and therefore has an even number of elements. It follows that
$\alpha(B')$ and $\omega(B')$ have different parity. 
Now by assumption $B$ and $B'$ are connected by a standard brace
at their endpoints; either they are connected at their left endpoints,
or at their right endpoints and in either case we conclude that
 $\alpha(B')$ is odd and $\omega(B')$ is even. This implies that
all  standard braces outside $I(B')$ are separated from the rest,
and since $\pi\in  \CompOdd{2n}$ it follows that $I(B')$ is the full intervall
$[2n]$, i.e.,  we have indeed type IIa or type IIb.
\end{proof}

As a corollary we obtain the following decomposition of flip partitions
which plays a major role in the  involution to be defined below. 
\begin{cor}
Any flip  partition $\pi$ can be decomposed as a disjoint union
 $$\pi=\nest(\BL)\cup \nest(\BR)\cup \{\BL\} \cup \{\BR\},$$ 
\end{cor}

The remaining odd partitions make up the last type.
\begin{defi}
  A partition $\pi\in \CompOdd{2n}$ which is not a flip partition
  is called type III. More specifically, it is type IIIa if
  the smallest element of the leftmost inner odd block is even,
  i.e., the left end point is the pivot element.
  It is  type IIIb if  the smallest element of the leftmost inner odd block is odd, i.e., the right end point is the pivot element.
\end{defi}

\subsection{Definition of the involution}
\label{ssec:involution}
We have now everything in place to define a sign-inverting involution on  $\CompOdd{2n}$, 
which simultaneously switches the sign of the corresponding term in the
expansion~\eqref{eq:commutatorexpansion}.
The involution acts on each type separately
and follows the patterns layed out in
Figures~\ref{fig:involutiontypeI}, \ref{fig:involutiontypeII} and~\ref{fig:involutiontypeIII}.

Types~I and~II are flip partitions
and the two odd blocks are flipped in such a way that
the decomposition \eqref{it:symetryczne:1} and \eqref{it:symetryczne:2} of
Lemma~\ref{lemma:symetryczne}, respectively, is preserved.
In type~III the pivot element is moved from one end of the leftmost
inner odd block to the other.
The braces are preserved except on the pivot brace, which is reversed.
Hereby the types are preserved, more precisely:

\begin{enumerate}[(I)]
 \item A partition of type I is mapped to type I, see Figures~\ref{fig:involutiontypeI}.
    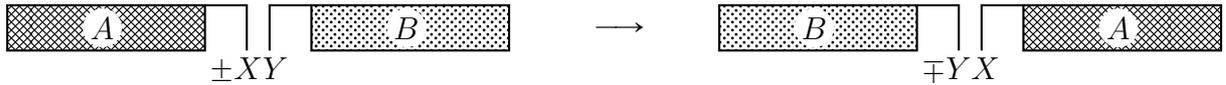
\begin{figure}[h]
    \begin{subfigure}[c]{0.35\textwidth}
      \centering
      \makeatletter{}
\begin{tikzpicture}[scale=0.5]
    \path [afill,thick] (0,0) rectangle (5.2,1.2);
    \path [bfill,thick] (8.0,0) rectangle (13.2,1.2);
\draw [thick] (5.2,1.2)--(6.3,1.2)--(6.3,0);
\draw [thick] (8.0,1.2)--(6.9,1.2)--(6.9,0);

\draw[white,fill=white] (2.5,0.6) circle (0.5); 
  \node at (2.5,0.6) {$A$};

\draw[white,fill=white] (10.5,0.6) circle (0.5); 
  \node at (10.5,0.6) {$B$};
\node at (6.35,-0.5){$\pm X Y$};

\node[color=white] at (1,2) {.};
\end{tikzpicture} 
    \end{subfigure}
    $\qquad\qquad\longrightarrow\qquad$
    \begin{subfigure}[c]{0.35\textwidth}
      \centering
      \makeatletter{}\begin{tikzpicture}[scale=0.5]
    \path [bfill,thick] (0,0) rectangle (5.2,1.2);
    \path [afill,thick] (8.0,0) rectangle (13.2,1.2);
\draw [thick] (5.2,1.2)--(6.3,1.2)--(6.3,0);
\draw [thick] (8.0,1.2)--(6.9,1.2)--(6.9,0);

\draw[white,fill=white] (2.5,0.6) circle (0.5); 
  \node at (2.5,0.6) {$B$};

\draw[white,fill=white] (10.5,0.6) circle (0.5); 
  \node at (10.5,0.6) {$A$};
\node at (6.35,-0.5){$\mp Y X$};

\node[color=white] at (1,2) {.};
\end{tikzpicture} 
    \end{subfigure}
    
    \caption{Involution of partitions of type I.} \label{fig:involutiontypeI}
  \end{figure}

  The length of the padding interval of the leftmost inner
  odd block (marked red in the diagram) is an odd number, say $2k+1$.
  Then the length of the padding interval of the other odd block is $2k'+1$, where $k'=r-k-1$.
  Then the intervals $[1,2,\dots, 2k]$ and $[2k+3,2k+4,\dots,2r]$ are
  flipped and the points $2k+1$ and $2k+2$ are exchanged;
  more precisely, the entries are mapped according to the action of the following permutation:
  \begin{equation}
    \sigma_{I,r,k}:
     i \mapsto
     \begin{cases}
       i+2k'+2 & \text{for $1\leq i\leq 2k$ }\\
       2k'+2   & \text{for $i=2k+1$}\\
       2k'+1   & \text{for $i=2k+2$}\\
       i-2k-2 & \text{for $2k+3\leq i\leq 2r$ }
     \end{cases}
     ;
  \end{equation}
  see Figure~\ref{fig:exinvolutiontypeI} for a specific example.
  It is easy to see that $\sigma_{I,r,k'}$ is the inverse of $\sigma_{I,r,k}$.
  \begin{figure}[h]
    \begin{subfigure}[c]{0.3\textwidth}
      \centering
      \makeatletter{}

\begin{tikzpicture}[scale=0.04]
\draw[style=aline] (2,0)--(2,7.5);
\draw[style=aline] (8,0)--(8,4.5);
\draw[style=aline] (14,0)--(14,4.5);
\draw[style=aline] (20,0)--(20,4.5);
\draw[style=aline] (26,0)--(26,4.5);
\draw[style=aline] (32,0)--(32,4.5);
\draw[style=aline] (38,0)--(38,4.5);
\draw[style=aline] (44,0)--(44,7.5);
\draw[style=aline] (2,7.5)--(44,7.5);
\draw[style=aline] (8,4.5)--(14,4.5);
\draw[style=aline] (20,4.5)--(38,4.5);

\draw (50,0)--(50,7.5)--(44,7.5);
\draw (56,0)--(56,7.5)--(62,7.5);

\draw[style=bline] (62,0)--(62,7.5);
\draw[style=bline] (68,0)--(68,4.5);
\draw[style=bline] (74,0)--(74,4.5);
\draw[style=bline] (80,0)--(80,4.5);
\draw[style=bline] (86,0)--(86,4.5);
\draw[style=bline] (92,0)--(92,4.5);
\draw[style=bline] (98,0)--(98,4.5);
\draw[style=bline] (104,0)--(104,4.5);
\draw[style=bline] (110,0)--(110,4.5);
\draw[style=bline] (116,0)--(116,7.5);
\draw[style=bline] (68,4.5)--(74,4.5);
\draw[style=bline] (80,4.5)--(86,4.5);
\draw[style=bline] (92,4.5)--(110,4.5);
\draw[style=bline] (62,7.5)--(116,7.5);

\draw (2,-9)--(2,-4.5);
\draw (8,-9)--(8,-4.5);
\draw (14,-9)--(14,-4.5);
\draw (20,-9)--(20,-4.5);
\draw (26,-9)--(26,-4.5);
\draw (32,-9)--(32,-4.5);
\draw (38,-9)--(38,-4.5);
\draw (44,-9)--(44,-4.5);
\draw[thick] (50,-9)--(50,-4.5);
\draw[thick] (56,-9)--(56,-4.5);
\draw (62,-9)--(62,-4.5);
\draw (68,-9)--(68,-4.5);
\draw (74,-9)--(74,-4.5);
\draw (80,-9)--(80,-4.5);
\draw (86,-9)--(86,-4.5);
\draw (92,-9)--(92,-4.5);
\draw (98,-9)--(98,-4.5);
\draw (104,-9)--(104,-4.5);
\draw (110,-9)--(110,-4.5);
\draw (116,-9)--(116,-4.5);
\draw (2,-4.5)--(8,-4.5);
\draw (14,-4.5)--(20,-4.5);
\draw (26,-4.5)--(32,-4.5);
\draw (38,-4.5)--(44,-4.5);
\draw[thick] (50,-4.5)--(56,-4.5);
\draw (62,-4.5)--(68,-4.5);
\draw (74,-4.5)--(80,-4.5);
\draw (86,-4.5)--(92,-4.5);
\draw (98,-4.5)--(104,-4.5);
\draw (110,-4.5)--(116,-4.5);

\node at (53,-15) {$\pm$};
\end{tikzpicture} 
    \end{subfigure}
    $\longrightarrow$
    \begin{subfigure}[c]{0.3\textwidth}
      \centering
      \makeatletter{}\begin{tikzpicture}[scale=0.04]
\draw[style=bline] (2,0)--(2,7.5);
\draw[style=bline] (8,0)--(8,4.5);
\draw[style=bline] (14,0)--(14,4.5);
\draw[style=bline] (20,0)--(20,4.5);
\draw[style=bline] (26,0)--(26,4.5);
\draw[style=bline] (32,0)--(32,4.5);
\draw[style=bline] (38,0)--(38,4.5);
\draw[style=bline] (44,0)--(44,4.5);
\draw[style=bline] (50,0)--(50,4.5);
\draw[style=bline] (56,0)--(56,7.5);

\draw (62,0)--(62,7.5)--(56,7.5);
\draw (68,0)--(68,7.5)--(74,7.5);

\draw[style=aline] (74,0)--(74,7.5);
\draw[style=aline] (80,0)--(80,4.5);
\draw[style=aline] (86,0)--(86,4.5);
\draw[style=aline] (92,0)--(92,4.5);
\draw[style=aline] (98,0)--(98,4.5);
\draw[style=aline] (104,0)--(104,4.5);
\draw[style=aline] (110,0)--(110,4.5);
\draw[style=aline] (116,0)--(116,7.5);

\draw[style=bline] (8,4.5)--(14,4.5);
\draw[style=bline] (20,4.5)--(26,4.5);
\draw[style=bline] (32,4.5)--(50,4.5);
\draw[style=bline] (2,7.5)--(56,7.5);

\draw[style=aline] (80,4.5)--(86,4.5);
\draw[style=aline] (92,4.5)--(110,4.5);
\draw[style=aline] (74,7.5)--(116,7.5);

\draw (2,-9)--(2,-4.5);
\draw (8,-9)--(8,-4.5);
\draw (14,-9)--(14,-4.5);
\draw (20,-9)--(20,-4.5);
\draw (26,-9)--(26,-4.5);
\draw (32,-9)--(32,-4.5);
\draw (38,-9)--(38,-4.5);
\draw (44,-9)--(44,-4.5);
\draw (50,-9)--(50,-4.5);
\draw (56,-9)--(56,-4.5);
\draw[thick] (62,-9)--(62,-4.5);
\draw[thick] (68,-9)--(68,-4.5);
\draw (74,-9)--(74,-4.5);
\draw (80,-9)--(80,-4.5);
\draw (86,-9)--(86,-4.5);
\draw (92,-9)--(92,-4.5);
\draw (98,-9)--(98,-4.5);
\draw (104,-9)--(104,-4.5);
\draw (110,-9)--(110,-4.5);
\draw (116,-9)--(116,-4.5);
\draw (2,-4.5)--(8,-4.5);
\draw (14,-4.5)--(20,-4.5);
\draw (26,-4.5)--(32,-4.5);
\draw (38,-4.5)--(44,-4.5);
\draw (50,-4.5)--(56,-4.5);
\draw[thick] (62,-4.5)--(68,-4.5);
\draw (74,-4.5)--(80,-4.5);
\draw (86,-4.5)--(92,-4.5);
\draw (98,-4.5)--(104,-4.5);
\draw (110,-4.5)--(116,-4.5);

\node at (65,-15) {$\mp$};

\end{tikzpicture}

    \end{subfigure}
    \vskip1ex
    \caption{An example of the involution of partitions of type I.} \label{fig:exinvolutiontypeI}
  \end{figure}
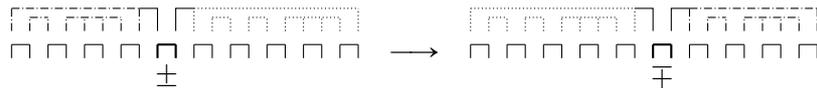

 \item type IIa is mapped to type IIb and vice versa, see
  Figure~\ref{fig:involutiontypeII}.
    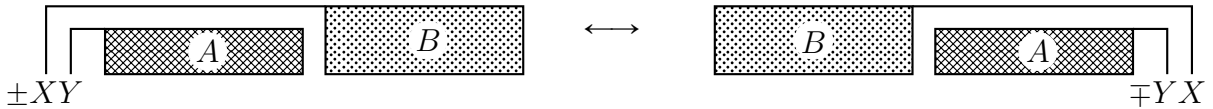
\begin{figure}[h]
    \begin{subfigure}[c]{0.35\textwidth}
      \centering
      \makeatletter{}
\begin{tikzpicture}[scale=0.5]
\draw [white] (0,3.5) circle (0);
    \path [afill,thick] (1.8,0) rectangle (7.0,1.2);
    \path [bfill,thick] (7.6,0) rectangle (12.8,1.8);
\draw [thick] (0.25,0)--(0.25,1.8)--(7.6,1.8);
\draw [thick] (0.9,0)--(0.9,1.2)--(1.8,1.2);

\draw[white,fill=white] (4.5,0.6) circle (0.5); 
  \node at (4.5,0.6) {$A$};

\draw[white,fill=white] (10.3,0.9) circle (0.5); 
  \node at (10.3,0.9) {$B$};
\node at (0.2,-0.5){$\pm X Y$};
\end{tikzpicture} 
    \end{subfigure}
    $\qquad\qquad\longleftrightarrow\qquad$
    \begin{subfigure}[c]{0.35\textwidth}
      \centering
      \makeatletter{}
\begin{tikzpicture}[scale=0.5]
\draw [white] (0,3.5) circle (0);
    \path [bfill,thick] (0,0) rectangle (5.2,1.8);
    \path [afill,thick] (5.8,0) rectangle (11.0,1.2);
\draw [thick] (12.55,0)--(12.55,1.8)--(5.2,1.8);
\draw [thick] (11.9,0)--(11.9,1.2)--(11.0,1.2);

\draw[white,fill=white] (2.6,0.9) circle (0.5); 
  \node at (2.6,0.85) {$B$};

\draw[white,fill=white] (8.5,0.6) circle (0.5); 
  \node at (8.5,0.6) {$A$};
\node at (11.9,-0.5){$\mp Y X$};
\end{tikzpicture} 
    \end{subfigure}
    \caption{Involution of partitions of type II.} \label{fig:involutiontypeII}
  \end{figure}

  The length of the padding interval of the leftmost inner
  odd block (marked red in the diagram) is an odd number, say $2k+1$ and let $k'=r-k-1$.
 
  In the case of type IIa the interval $[3,4,\dots ,2k]$ is flipped with the interval
  $[2k+1,2k+2,\dots,2r]$ and the pair $(1,2)$ is mapped to the pair $(2r,2r-1)$
  (notice the change of order);   more precisely, the entries are mapped according to the action of the following permutation:
  \begin{equation}
    \sigma_{IIa,r,k}:
     i \mapsto
     \begin{cases}
       2r   & \text{for $i=1$}\\
       2r-1   & \text{for $i=2$}\\       
       i + 2k' & \text{for $3\leq i\leq 2k$ }\\
       i-2k-2 & \text{for $2k+3\leq i\leq 2r$ }
     \end{cases}
     ;
  \end{equation}
  see Figure~\ref{fig:exinvolutiontypeII} for a specific example.

  In the case of type IIb we reverse the above process.
  Now the interval $[1,2,\dots ,2k]$ is flipped with the interval
  $[2k+1,2k+2,\dots,2r-2]$ and the pair $(2r,2r-1)$ is mapped to the pair  $(1,2)$;
  more precisely, the entries are mapped according to the action of the following permutation:
  \begin{equation}
    \sigma_{IIb,r,k}:
     i \mapsto
     \begin{cases}
       i + 2k'+2 & \text{for $1\leq i\leq 2k$ }\\
       i-2k+2 & \text{for $2k+1\leq i\leq 2r-2$ }\\
       2   & \text{for $i=2r-1$}\\
       1   & \text{for $i=2r$}
     \end{cases}
     .
  \end{equation}

  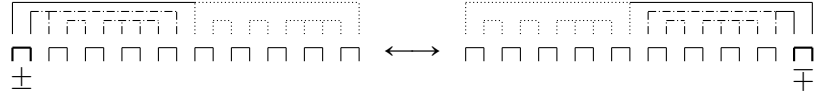
\begin{figure}[h]
    \begin{subfigure}[c]{0.3\textwidth}
      \centering
      \makeatletter{}
\begin{tikzpicture}[scale=0.04]
\draw (2,0)--(2,10.5);
\draw (8,0)--(8,7.5);
\draw[style=aline] (14,0)--(14,7.5);
\draw[style=aline] (20,0)--(20,4.5);
\draw[style=aline] (26,0)--(26,4.5);
\draw[style=aline] (32,0)--(32,4.5);
\draw[style=aline] (38,0)--(38,4.5);
\draw[style=aline] (44,0)--(44,4.5);
\draw[style=aline] (50,0)--(50,4.5);
\draw[style=aline] (56,0)--(56,7.5);
\draw[style=aline] (14,7.5)--(56,7.5);
\draw[style=aline] (20,4.5)--(26,4.5);
\draw[style=aline] (32,4.5)--(50,4.5);
\draw[style=aline] (8,7.5)--(14,7.5);

\draw[style=bline] (62,0)--(62,10.5);
\draw[style=bline] (68,0)--(68,4.5);
\draw[style=bline] (74,0)--(74,4.5);
\draw[style=bline] (80,0)--(80,4.5);
\draw[style=bline] (86,0)--(86,4.5);
\draw[style=bline] (92,0)--(92,4.5);
\draw[style=bline] (98,0)--(98,4.5);
\draw[style=bline] (104,0)--(104,4.5);
\draw[style=bline] (110,0)--(110,4.5);
\draw[style=bline] (116,0)--(116,10.5);
\draw[style=bline] (68,4.5)--(74,4.5);
\draw[style=bline] (80,4.5)--(86,4.5);
\draw[style=bline] (92,4.5)--(110,4.5);
\draw[style=bline] (62,10.5)--(116,10.5);

\draw (2,10.5)--(62,10.5);

\draw[thick] (2,-9)--(2,-4.5);
\draw[thick] (8,-9)--(8,-4.5);
\draw (14,-9)--(14,-4.5);
\draw (20,-9)--(20,-4.5);
\draw (26,-9)--(26,-4.5);
\draw (32,-9)--(32,-4.5);
\draw (38,-9)--(38,-4.5);
\draw (44,-9)--(44,-4.5);
\draw (50,-9)--(50,-4.5);
\draw (56,-9)--(56,-4.5);
\draw (62,-9)--(62,-4.5);
\draw (68,-9)--(68,-4.5);
\draw (74,-9)--(74,-4.5);
\draw (80,-9)--(80,-4.5);
\draw (86,-9)--(86,-4.5);
\draw (92,-9)--(92,-4.5);
\draw (98,-9)--(98,-4.5);
\draw (104,-9)--(104,-4.5);
\draw (110,-9)--(110,-4.5);
\draw (116,-9)--(116,-4.5);
\draw[thick] (2,-4.5)--(8,-4.5);
\draw (14,-4.5)--(20,-4.5);
\draw (26,-4.5)--(32,-4.5);
\draw (38,-4.5)--(44,-4.5);
\draw (50,-4.5)--(56,-4.5);
\draw (62,-4.5)--(68,-4.5);
\draw (74,-4.5)--(80,-4.5);
\draw (86,-4.5)--(92,-4.5);
\draw (98,-4.5)--(104,-4.5);
\draw (110,-4.5)--(116,-4.5);

\node at (5,-15) {$\pm$};
\end{tikzpicture}
 
    \end{subfigure}
    $\longleftrightarrow$
    \begin{subfigure}[c]{0.3\textwidth}
      \centering
      \makeatletter{}\begin{tikzpicture}[scale=0.04]
\draw[style=bline] (2,0)--(2,10.5);
\draw[style=bline] (8,0)--(8,4.5);
\draw[style=bline] (14,0)--(14,4.5);
\draw[style=bline] (20,0)--(20,4.5);
\draw[style=bline] (26,0)--(26,4.5);
\draw[style=bline] (32,0)--(32,4.5);
\draw[style=bline] (38,0)--(38,4.5);
\draw[style=bline] (44,0)--(44,4.5);
\draw[style=bline] (50,0)--(50,4.5);
\draw[style=bline] (56,0)--(56,10.5);
\draw[style=bline] (2,10.5)--(56,10.5);

\draw[style=bline] (8,4.5)--(14,4.5);
\draw[style=bline] (20,4.5)--(26,4.5);
\draw[style=bline] (32,4.5)--(50,4.5);

\draw (56,10.5)--(116,10.5);

\draw[style=aline] (62,0)--(62,7.5);
\draw[style=aline] (68,0)--(68,4.5);
\draw[style=aline] (74,0)--(74,4.5);
\draw[style=aline] (80,0)--(80,4.5);
\draw[style=aline] (86,0)--(86,4.5);
\draw[style=aline] (92,0)--(92,4.5);
\draw[style=aline] (98,0)--(98,4.5);
\draw[style=aline] (104,0)--(104,7.5);
\draw[style=aline] (62,7.5)--(104,7.5);

\draw[style=aline] (68,4.5)--(74,4.5);
\draw[style=aline] (80,4.5)--(98,4.5);

\draw (104,7.5)--(110,7.5);

\draw (110,0)--(110,7.5);
\draw (116,0)--(116,10.5);

\draw (2,-9)--(2,-4.5);
\draw (8,-9)--(8,-4.5);
\draw (14,-9)--(14,-4.5);
\draw (20,-9)--(20,-4.5);
\draw (26,-9)--(26,-4.5);
\draw (32,-9)--(32,-4.5);
\draw (38,-9)--(38,-4.5);
\draw (44,-9)--(44,-4.5);
\draw (50,-9)--(50,-4.5);
\draw (56,-9)--(56,-4.5);
\draw (62,-9)--(62,-4.5);
\draw (68,-9)--(68,-4.5);
\draw (74,-9)--(74,-4.5);
\draw (80,-9)--(80,-4.5);
\draw (86,-9)--(86,-4.5);
\draw (92,-9)--(92,-4.5);
\draw (98,-9)--(98,-4.5);
\draw (104,-9)--(104,-4.5);
\draw[thick] (110,-9)--(110,-4.5);
\draw[thick] (116,-9)--(116,-4.5);

\draw (2,-4.5)--(8,-4.5);
\draw (14,-4.5)--(20,-4.5);
\draw (26,-4.5)--(32,-4.5);
\draw (38,-4.5)--(44,-4.5);
\draw (50,-4.5)--(56,-4.5);
\draw (62,-4.5)--(68,-4.5);
\draw (74,-4.5)--(80,-4.5);
\draw (86,-4.5)--(92,-4.5);
\draw (98,-4.5)--(104,-4.5);
\draw[thick] (110,-4.5)--(116,-4.5);

\node at (113,-15) {$\mp$};

\end{tikzpicture}
 
    \end{subfigure}
    \vskip1ex
    \caption{An example of the involution of partitions of type II.} \label{fig:exinvolutiontypeII}
  \end{figure}

 \item type IIIa is mapped to type IIIb and vice versa, see
  Figure~\ref{fig:involutiontypeIII}.
    \begin{figure}[h]

    \begin{subfigure}[c]{0.35\textwidth}
      \centering
      \makeatletter{}
\begin{tikzpicture}[scale=0.5]
\draw [white] (0,3.0) circle (0);

    \path [afill,thick] (2.0,0) rectangle (7.2,1.0);
\draw [thick] (0.7,0)--(0.7,1.6)--(7.9,1.6)--(7.9,0);
\draw [thick] (1.3,0)--(1.3,1.0)--(2.0,1.0);
\draw [dashed] (-1.8,1.6)--(10.3,1.6);

\node at (0.6,-0.5){$\pm X Y$};
\end{tikzpicture} 
    \end{subfigure}
    $\qquad\quad\longleftrightarrow\qquad$
    \begin{subfigure}[c]{0.35\textwidth}
      \centering
      \makeatletter{}
\begin{tikzpicture}[scale=0.5]
\draw [white] (0,3.0) circle (0);
    \path [afill,thick] (1.1,0) rectangle (6.3,1.0);
\draw [thick] (0.7,0)--(0.7,1.6)--(7.9,1.6)--(7.9,0);
\draw [thick] (7.2,0)--(7.2,1.0)--(6.3,1.0);
\draw [dashed] (-1.8,1.6)--(10.3,1.6);

\node at (7.2,-0.5){$\mp Y X$};
\end{tikzpicture} 
    \end{subfigure}

    \caption{Involution of partitions of type III.} \label{fig:involutiontypeIII}
  \end{figure}
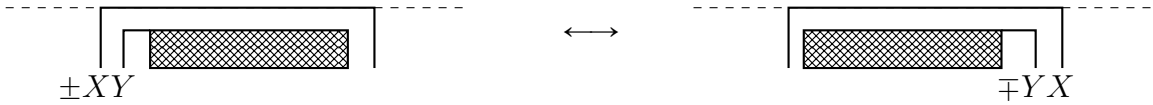
\end{enumerate}

In the case of type III we apply a rotation to the padding interval of
the leftmost inner odd block augmented by missing element from the pivot brace.
More precisely, if in type IIIa the leftmost inner odd block starts at $2k$
and ends at $2l$ then the permutation is the square of the  cycle
spanned by its padding interval together with the pivot brace:
$$
\sigma_{IIIa,r,k,l} = 
(2l,2l-1,\dots,2k-1)^2=
(2l,2l-2,\dots,2k)\circ (2l-1,2l-3,\dots,2k-1)
$$
conversely, if in type IIIb 
the leftmost inner odd block starts at $2k+1$
and ends at $2l-1$ then the permutation is the square of the corresponding cycle
$$
\sigma_{IIIb,r,k,l} = 
(2k+1,2k+2,\dots,2l)^2
=
(2k,2k+2,\dots,2l)\circ (2k+1,2k+3,\dots,2l-1)
;
$$
  see Figure~\ref{fig:exinvolutiontypeIII} for specific examples.
  \begin{figure}[h]
    \begin{subfigure}[c]{0.3\textwidth}
      \centering
      \makeatletter{}\begin{tikzpicture}[scale=0.04]
\draw (2,0)--(2,10.5);
\draw (8,0)--(8,4.5);
\draw (14,0)--(14,4.5);
\draw (20,0)--(20,10.5);
\draw (26,0)--(26,10.5);
\draw (32,0)--(32,7.5);

\draw[style=aline] (38,0)--(38,7.5);
\draw[style=aline] (44,0)--(44,4.5);
\draw[style=aline] (50,0)--(50,4.5);
\draw[style=aline] (56,0)--(56,4.5);
\draw[style=aline] (62,0)--(62,4.5);
\draw[style=aline] (68,0)--(68,4.5);
\draw[style=aline] (74,0)--(74,4.5);
\draw[style=aline] (80,0)--(80,7.5);
\draw[style=aline] (44,4.5)--(50,4.5);
\draw[style=aline] (56,4.5)--(74,4.5);
\draw[style=aline] (38,7.5)--(80,7.5);

\draw (86,0)--(86,10.5);
\draw (92,0)--(92,4.5);
\draw (98,0)--(98,4.5);
\draw (104,0)--(104,4.5);
\draw (110,0)--(110,4.5);
\draw (116,0)--(116,10.5);
\draw (8,4.5)--(14,4.5);
\draw (32,7.5)--(38,7.5);
\draw (92,4.5)--(110,4.5);
\draw (2,10.5)--(116,10.5);

\draw (2,-9)--(2,-4.5);
\draw (8,-9)--(8,-4.5);
\draw (14,-9)--(14,-4.5);
\draw (20,-9)--(20,-4.5);
\draw[thick] (26,-9)--(26,-4.5);
\draw[thick] (32,-9)--(32,-4.5);
\draw (38,-9)--(38,-4.5);
\draw (44,-9)--(44,-4.5);
\draw (50,-9)--(50,-4.5);
\draw (56,-9)--(56,-4.5);
\draw (62,-9)--(62,-4.5);
\draw (68,-9)--(68,-4.5);
\draw (74,-9)--(74,-4.5);
\draw (80,-9)--(80,-4.5);
\draw (86,-9)--(86,-4.5);
\draw (92,-9)--(92,-4.5);
\draw (98,-9)--(98,-4.5);
\draw (104,-9)--(104,-4.5);
\draw (110,-9)--(110,-4.5);
\draw (116,-9)--(116,-4.5);
\draw (2,-4.5)--(8,-4.5);
\draw (14,-4.5)--(20,-4.5);
\draw[thick] (26,-4.5)--(32,-4.5);
\draw (38,-4.5)--(44,-4.5);
\draw (50,-4.5)--(56,-4.5);
\draw (62,-4.5)--(68,-4.5);
\draw (74,-4.5)--(80,-4.5);
\draw (86,-4.5)--(92,-4.5);
\draw (98,-4.5)--(104,-4.5);
\draw (110,-4.5)--(116,-4.5);

\node at (29,-15) {$\pm$};
\end{tikzpicture}       
    \end{subfigure}
    $\longleftrightarrow$
    \begin{subfigure}[c]{0.3\textwidth}
      \centering
      \makeatletter{}\begin{tikzpicture}[scale=0.04]
\draw (2,0)--(2,10.5);
\draw (8,0)--(8,4.5);
\draw (14,0)--(14,4.5);
\draw (20,0)--(20,10.5);

\draw[style=aline] (26,0)--(26,7.5);
\draw[style=aline] (32,0)--(32,4.5);
\draw[style=aline] (38,0)--(38,4.5);
\draw[style=aline] (44,0)--(44,4.5);
\draw[style=aline] (50,0)--(50,4.5);
\draw[style=aline] (56,0)--(56,4.5);
\draw[style=aline] (62,0)--(62,4.5);
\draw[style=aline] (68,0)--(68,7.5);
\draw[style=aline] (26,7.5)--(68,7.5);
\draw[style=aline] (32,4.5)--(38,4.5);
\draw[style=aline] (44,4.5)--(62,4.5);

\draw (68,7.5)--(74,7.5);

\draw (74,0)--(74,7.5);
\draw (80,0)--(80,10.5);
\draw (86,0)--(86,10.5);
\draw (92,0)--(92,4.5);
\draw (98,0)--(98,4.5);
\draw (104,0)--(104,4.5);
\draw (110,0)--(110,4.5);
\draw (116,0)--(116,10.5);
\draw (8,4.5)--(14,4.5);
\draw (92,4.5)--(110,4.5);
\draw (2,10.5)--(116,10.5);

\draw (2,-9)--(2,-4.5);
\draw (8,-9)--(8,-4.5);
\draw (14,-9)--(14,-4.5);
\draw (20,-9)--(20,-4.5);
\draw (26,-9)--(26,-4.5);
\draw (32,-9)--(32,-4.5);
\draw (38,-9)--(38,-4.5);
\draw (44,-9)--(44,-4.5);
\draw (50,-9)--(50,-4.5);
\draw (56,-9)--(56,-4.5);
\draw (62,-9)--(62,-4.5);
\draw (68,-9)--(68,-4.5);
\draw[thick] (74,-9)--(74,-4.5);
\draw[thick] (80,-9)--(80,-4.5);
\draw (86,-9)--(86,-4.5);
\draw (92,-9)--(92,-4.5);
\draw (98,-9)--(98,-4.5);
\draw (104,-9)--(104,-4.5);
\draw (110,-9)--(110,-4.5);
\draw (116,-9)--(116,-4.5);
\draw (2,-4.5)--(8,-4.5);
\draw (14,-4.5)--(20,-4.5);
\draw (26,-4.5)--(32,-4.5);
\draw (38,-4.5)--(44,-4.5);
\draw (50,-4.5)--(56,-4.5);
\draw (62,-4.5)--(68,-4.5);
\draw[thick] (74,-4.5)--(80,-4.5);
\draw (86,-4.5)--(92,-4.5);
\draw (98,-4.5)--(104,-4.5);
\draw (110,-4.5)--(116,-4.5);

\node at (77,-15) {$\mp$};
\end{tikzpicture}       
    \end{subfigure}
    \vskip5ex

    \begin{subfigure}[c]{0.3\textwidth}
      \centering
      \makeatletter{}
\begin{tikzpicture}[scale=0.04]
\draw (2,0)--(2,10.5);
\draw (8,0)--(8,7.5);

\draw[style=aline] (14,0)--(14,7.5);
\draw[style=aline] (20,0)--(20,4.5);
\draw[style=aline] (26,0)--(26,4.5);
\draw[style=aline] (32,0)--(32,4.5);
\draw[style=aline] (38,0)--(38,4.5);
\draw[style=aline] (44,0)--(44,4.5);
\draw[style=aline] (50,0)--(50,4.5);
\draw[style=aline] (56,0)--(56,7.5);
\draw[style=aline] (14,7.5)--(56,7.5);
\draw[style=aline] (20,4.5)--(26,4.5);
\draw[style=aline] (32,4.5)--(50,4.5);

\draw (62,0)--(62,10.5);
\draw (68,0)--(68,10.5);
\draw (74,0)--(74,7.5);
\draw (80,0)--(80,4.5);
\draw (86,0)--(86,4.5);
\draw (92,0)--(92,7.5);
\draw (98,0)--(98,7.5);
\draw (104,0)--(104,4.5);
\draw (110,0)--(110,4.5);
\draw (116,0)--(116,10.5);

\draw (8,7.5)--(14,7.5);
\draw (2,10.5)--(62,10.5);

\draw (80,4.5)--(86,4.5);
\draw (74,7.5)--(98,7.5);
\draw (104,4.5)--(110,4.5);
\draw (68,10.5)--(116,10.5);

\draw[thick] (2,-9)--(2,-4.5);
\draw[thick] (8,-9)--(8,-4.5);
\draw (14,-9)--(14,-4.5);
\draw (20,-9)--(20,-4.5);
\draw (26,-9)--(26,-4.5);
\draw (32,-9)--(32,-4.5);
\draw (38,-9)--(38,-4.5);
\draw (44,-9)--(44,-4.5);
\draw (50,-9)--(50,-4.5);
\draw (56,-9)--(56,-4.5);
\draw (62,-9)--(62,-4.5);
\draw (68,-9)--(68,-4.5);
\draw (74,-9)--(74,-4.5);
\draw (80,-9)--(80,-4.5);
\draw (86,-9)--(86,-4.5);
\draw (92,-9)--(92,-4.5);
\draw (98,-9)--(98,-4.5);
\draw (104,-9)--(104,-4.5);
\draw (110,-9)--(110,-4.5);
\draw (116,-9)--(116,-4.5);
\draw[thick] (2,-4.5)--(8,-4.5);
\draw (14,-4.5)--(20,-4.5);
\draw (26,-4.5)--(32,-4.5);
\draw (38,-4.5)--(44,-4.5);
\draw (50,-4.5)--(56,-4.5);
\draw (62,-4.5)--(68,-4.5);
\draw (74,-4.5)--(80,-4.5);
\draw (86,-4.5)--(92,-4.5);
\draw (98,-4.5)--(104,-4.5);
\draw (110,-4.5)--(116,-4.5);

\node at (5,-15) {$\pm$};
\end{tikzpicture}
 
    \end{subfigure}
    $\longleftrightarrow$
    \begin{subfigure}[c]{0.3\textwidth}
      \centering
      \makeatletter{}
\begin{tikzpicture}[scale=0.04]
\draw[style=aline] (2,0)--(2,7.5);
\draw[style=aline] (8,0)--(8,4.5);
\draw[style=aline] (14,0)--(14,4.5);
\draw[style=aline] (20,0)--(20,4.5);
\draw[style=aline] (26,0)--(26,4.5);
\draw[style=aline] (32,0)--(32,4.5);
\draw[style=aline] (38,0)--(38,4.5);
\draw[style=aline] (44,0)--(44,7.5);
\draw[style=aline] (2,7.5)--(44,7.5);
\draw[style=aline] (8,4.5)--(14,4.5);
\draw[style=aline] (20,4.5)--(38,4.5);

\draw (50,0)--(50,7.5);
\draw (56,0)--(56,4.5);
\draw (62,0)--(62,4.5);
\draw (68,0)--(68,10.5);
\draw (74,0)--(74,7.5);
\draw (80,0)--(80,4.5);
\draw (86,0)--(86,4.5);
\draw (92,0)--(92,7.5);
\draw (98,0)--(98,7.5);
\draw (104,0)--(104,4.5);
\draw (110,0)--(110,4.5);
\draw (116,0)--(116,10.5);

\draw (44,7.5)--(50,7.5);
\draw (56,4.5)--(62,4.5);
\draw (80,4.5)--(86,4.5);
\draw (74,7.5)--(98,7.5);
\draw (104,4.5)--(110,4.5);
\draw (68,10.5)--(116,10.5);

\draw (2,-9)--(2,-4.5);
\draw (8,-9)--(8,-4.5);
\draw (14,-9)--(14,-4.5);
\draw (20,-9)--(20,-4.5);
\draw (26,-9)--(26,-4.5);
\draw (32,-9)--(32,-4.5);
\draw (38,-9)--(38,-4.5);
\draw (44,-9)--(44,-4.5);
\draw[thick] (50,-9)--(50,-4.5);
\draw[thick] (56,-9)--(56,-4.5);
\draw (62,-9)--(62,-4.5);
\draw (68,-9)--(68,-4.5);
\draw (74,-9)--(74,-4.5);
\draw (80,-9)--(80,-4.5);
\draw (86,-9)--(86,-4.5);
\draw (92,-9)--(92,-4.5);
\draw (98,-9)--(98,-4.5);
\draw (104,-9)--(104,-4.5);
\draw (110,-9)--(110,-4.5);
\draw (116,-9)--(116,-4.5);
\draw (2,-4.5)--(8,-4.5);
\draw (14,-4.5)--(20,-4.5);
\draw (26,-4.5)--(32,-4.5);
\draw (38,-4.5)--(44,-4.5);
\draw[thick] (50,-4.5)--(56,-4.5);
\draw (62,-4.5)--(68,-4.5);
\draw (74,-4.5)--(80,-4.5);
\draw (86,-4.5)--(92,-4.5);
\draw (98,-4.5)--(104,-4.5);
\draw (110,-4.5)--(116,-4.5);

\node at (53,-15) {$\mp$};

\end{tikzpicture}
 
    \end{subfigure}
    \vskip1ex

    \caption{Examples of the involution of partitions of type III.} \label{fig:exinvolutiontypeIII}
  \end{figure}
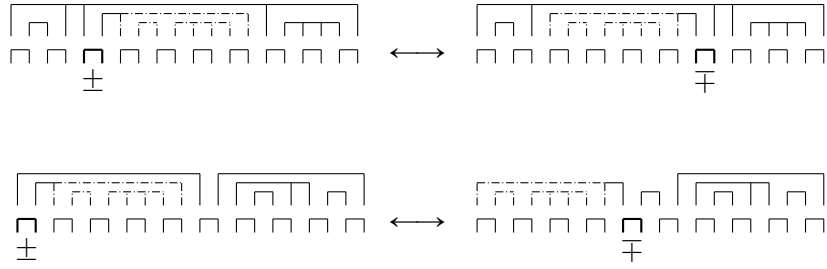

  \begin{prop}
    \begin{enumerate}[(i)]
     \item 
      The previously constructed permutations are inverse to each other, more
      precisely:
      \begin{align*}
        \sigma_{I,r,k} &=\sigma_{I,r,r-k-1}^{-1}\\
        \sigma_{IIa,r,k} &=\sigma_{IIb,r,r-k-1}^{-1}\\
        \sigma_{IIIa,k,l} &=\sigma_{IIIb,l,k}^{-1}
      \end{align*}
     \item  For a flip partition $\pi\in \CompOdd{2r}$ let us denote by
      $\sigma_\pi$ the permutation constructed above.
      Then the map
      \begin{align*}
        \psi : \CompOdd{2r}&\to\CompOdd{2r}\\
        \pi&\mapsto\sigma_\pi\cdot\pi
      \end{align*}
      is an involution.
    \end{enumerate}
  \end{prop}
  \begin{proof}
    Part (i) is immediate.

    To see part (ii) we first observe that each type is mapped onto itself.
    In type I  the map is obviously involutive; as for type II and III,
    the role of the innermost odd block is left invariant and thus we have
    indeed an involution. 
  \end{proof}
 \begin{Rem}
  \begin{enumerate}
   \item  In type I and {II} we flip blocks  and the  pivot brace $\SpPair_\pi$
    is reversed, which will imply a change of sign and thus a cancellation in the formulas below;
    the odd blocks are flipped appropriately and the remaining
    blocks in this decomposition  are shifted
    but the tracial structure is unchanged.

In type {III}  we shift and flip the pivot brace $\SpPair_\pi$ only; the
remaining blocks stay in place.
During this procedure we also rotate  appropriately the two  points of the
blocks  $\BL$ and $\BR$ which are joined by $\SpPair_\pi$.
 Otherwise the structure of this  block is not changed.  
 
\item  If $\sigma:\pi \mapsto \pi'$, then $\pi$ and $\pi'$ have the same block structure i.e., for every $1 \leq m \leq n$,
the two partitions $\pi$ and $\pi'$ have the same number of blocks with $m$ elements. Indeed observe that our condition just say that we remove one point and add  one point or shift corresponding blocks. During this procedure we can not change  framework of block.

\item
 The following example shows
that  we cannot apply the rules of type~{III} universally without losing
the involutive property. Namely, applying the rule of type~{III} to the
following type~{II} partition twice we obtain:
\begin{equation}
  \label{eq:typeIIbad}
  \begin{array}{ccccc}
\makeatletter{}\begin{tikzpicture}[scale=0.07]
\draw (2,0)--(2,7.5);
\draw (8,0)--(8,7.5);
\draw (14,0)--(14,4.5);
\draw (20,0)--(20,4.5);
\draw (26,0)--(26,4.5);
\draw (32,0)--(32,7.5);
\draw (14,4.5)--(26,4.5);
\draw (2,7.5)--(32,7.5);

\draw (2,-9)--(2,-4.5);
\draw (8,-9)--(8,-4.5);
\draw (14,-9)--(14,-4.5);
\draw (20,-9)--(20,-4.5);
\draw[thline] (26,-9)--(26,-4.5);
\draw[thline] (32,-9)--(32,-4.5);
\draw (2,-4.5)--(8,-4.5);
\draw (14,-4.5)--(20,-4.5);
\draw[thline] (26,-4.5)--(32,-4.5);
\end{tikzpicture}  &    \to & 
\makeatletter{}\begin{tikzpicture}[scale=0.07]
\draw (2,0)--(2,4.5);
\draw (8,0)--(8,4.5);
\draw (14,0)--(14,4.5);
\draw (20,0)--(20,4.5);
\draw (26,0)--(26,4.5);
\draw (32,0)--(32,4.5);
\draw (2,4.5)--(14,4.5);
\draw (20,4.5)--(32,4.5);

\draw (2,-9)--(2,-4.5);
\draw (8,-9)--(8,-4.5);
\draw[thline] (14,-9)--(14,-4.5);
\draw[thline] (20,-9)--(20,-4.5);
\draw (26,-9)--(26,-4.5);
\draw (32,-9)--(32,-4.5);
\draw (2,-4.5)--(8,-4.5);
\draw[thline] (14,-4.5)--(20,-4.5);
\draw (26,-4.5)--(32,-4.5);
\end{tikzpicture}
 
 & \to &
\makeatletter{}\begin{tikzpicture}[scale=0.07]
\draw (2,0)--(2,7.5);
\draw (8,0)--(8,4.5);
\draw (14,0)--(14,4.5);
\draw (20,0)--(20,4.5);
\draw (26,0)--(26,7.5);
\draw (32,0)--(32,7.5);
\draw (8,4.5)--(20,4.5);
\draw (2,7.5)--(32,7.5);

\draw[thline] (2,-9)--(2,-4.5);
\draw[thline] (8,-9)--(8,-4.5);
\draw (14,-9)--(14,-4.5);
\draw (20,-9)--(20,-4.5);
\draw (26,-9)--(26,-4.5);
\draw (32,-9)--(32,-4.5);
\draw[thline] (2,-4.5)--(8,-4.5);
\draw (14,-4.5)--(20,-4.5);
\draw (26,-4.5)--(32,-4.5);
\end{tikzpicture}
 
\\
    \{(1,2,6),(3,4,5)\} && \{(1,2,3),(4,5,6)\} && \{(2,3,4),(1,5,6)\}
  \end{array}
\end{equation}
  \end{enumerate}

\end{Rem}

\section{Preservation of infinite divisibility and cancellation of odd cumulants}

\subsection{The main result}
In our previous paper we observed that the phenomenon
of \emph{cancellation of odd cumulants} to be defined below is related to 
preservation of free infinite divisibility.

\begin{theo}[{\cite[{Corollary~4.14}]{EjsmontLehner:2017}}]
  \label{twr:niskonczonapodzielnoscEL}
Let $\X_1, \X_2,\dots, \X_n\in \A_{sa}$ be a free family  
  of  $\boxplus$-infinitely divisible random variables. 
  Let $P$ be a selfadjoint  polynomial of degree $2$
  in noncommuting variables which exhibits cancellation of odd cumulants.
  Then the distribution of\/  $Y$ is $\boxplus$-infinitely divisible as well.
\end{theo}
In \cite[{Conjecture 5.2}]{EjsmontLehner:2017}
we also conjectured that the cancellation phenomenon is actually equivalent to
preservation of infinite divisibility. 
In this section we confirm this conjecture for quadratic forms and it turns out
that the equivalence can be extended to several other properties.
\begin{defi}
  Let $Q_n = \sum_{i,j=1}^n a_{i,j}X_iX_j$ be a quadratic form in noncommuting
  variables $X_1$,$X_2$,\ldots{},$X_n$  with system matrix
  $A=[a_{i,j}]_{i,j=1}^n\in M_n(\IC)$,
  which is assumed to be selfadjoint.
  When the formal variables $X_i$ are replaced by (noncommutative) selfadjoint random
  variables,   then  $Q_n$ becomes a selfadjoint random variable as well.
  With the aid of the product formula~\eqref{twr:produktargumentow}
  a universal formula for its free cumulants can be computed, involving the
  coefficients $a_{i,j}$ and the joint cumulants of the random variables $X_i$.
    \begin{enumerate}[(i)]
     \item \label{defit:cancellation} We say that $Q_n$ exhibits \emph{free cancellation of odd cumulants}
      if for any free family of selfadjoint noncommutative random variables
      $X_1,X_2,\dots,X_n$ 
      the odd cumulants of the $X_i$ do not contribute to the universal
      formula for the cumulants of $Q_n$.
     \item \label{defit:strongcancellation} We say that $Q_n$ exhibits \emph{strong cancellation of odd cumulants}
      if for any family of selfadjoint noncommutative random variables
      $X_1,X_2,\dots,X_n$ in a tracial noncommutative probability space
      the odd joint cumulants of the $X_i$ do not contribute to the universal
      formula for the cumulants of $Q_n$.
     \item We say that $Q_n$ \emph{preserves infinite divisibility} if for any
      free family of selfadjoint noncommutative random variables $X_1,X_2,\dots,X_n$ 
      with freely infinitely divisible laws, the law of $Q_n$ is also freely infinitely divisible.
  \end{enumerate}

\end{defi}

\begin{Rem}
  \begin{enumerate}[1.]
   \item   We would like to emphasize that  freeness is not assumed in
  condition~\eqref{defit:strongcancellation}, which  asserts that
  in addition to the univariate odd cumulants also the mixed odd cumulants cancel.
 \item 
  Free and strong cancellation of odd cumulants are not
  equivalent for higher order polynomials.
  For example, it is immediate that the iterated commutator $[[X_1,X_2],X_3]$ 
  exhibits
  free cancellation of odd cumulants, but if $X_1$ and $X_2$ 
  are identically distributed with free cumulants $r_1,r_2,r_3,\dots$ 
  and $X_3$ is replaced by $X_1$, 
  a calculation shows that the third cumulant
  $$
  K_3( [[X_1,X_2],X_1]) = -6 r_2 r_3 r_4+6 r_3^3 -6 r_2^3 r_3 
  $$
  depends on the third cumulant $r_3$.
 \item 
  Concerning item~\eqref{defit:cancellation} we remark that
  any free family of selfadjoint
  variables can be realized in a tracial probability space
  \cite[Proposition~2.5.3.]{VoiculescuDykemaNica:1992} and therefore
  the traciality condition is implicitly satisfied.
  \end{enumerate}
\end{Rem}
We call a selfadjoint  matrix   $A\in M_n(\IC)$  \emph{skew symmetric} if  $A=-A^T$; that is, $A=i\tilde{A}$, where $\tilde{A}$ is a real skew symmetric
  matrix (this is in order to distinguish between skew-Hermitian matrix).

\begin{theo} \label{prop:znikaniekumulant}
  The following properties are equivalent for a quadratic form
  $T_n=\sum_{i,j=1}^na_{i,j}\X_i\X_j$
  with selfadjoint system matrix   $A=[a_{i,j}]_{i,j=1}^n\in M_n(\IC)$.
\begin{enumerate}[(i)]
 \item \label{main:it:strongcancel}   $T_n$ exhibits  strong cancellation of odd cumulants.
 \item \label{main:it:cancel}   $T_n$ exhibits  cancellation of odd cumulants.
 \item \label{main:it:preserve}   $T_n$ preserves free infinite divisibility.
 \item \label{main:it:antisym}     $A$ is skew symmetric or equivalently, $T_n=\sum_{k<l} a_{k,l} (X_kX_l-X_lX_k)$ is a sum of commutators.
 \item \label{main:it:sym}   The distribution of 
  $T_n$ is symmetric for any free family of selfadjoint random variables $\X_1, \X_2,\dots, \X_n$.
\end{enumerate}
\end{theo}

The crucial steps are the implications
\eqref{main:it:cancel}$\implies{}$\eqref{main:it:preserve}
, \eqref{main:it:antisym}$\implies{}$\eqref{main:it:strongcancel} and \eqref{main:it:cancel}+\eqref{main:it:antisym}$\implies{}$\eqref{main:it:sym}.
The former is the content of Theorem~\ref{twr:niskonczonapodzielnoscEL} and
for the latter we will apply the involution of the previous section in
combination with the following lemma.
\begin{lemm}
  \label{lem:involution}
  \begin{enumerate}[(i)]
   \item 
    Let $\pi$ be a partition of type I, II or III and $\sigma$ the corresponding permutation
    constructed in section~\ref{ssec:involution}. Then for elements
    $X_1,X_2,\dots,X_n$ of a tracial probability space the cumulant is invariant:
    $$
    K_\pi(X_{i_1},X_{i_2},\dots,X_{i_{2r}}) =
        K_{\sigma\cdot\pi}(X_{i_{\sigma(1)}},X_{i_{\sigma(2)}},\dots,X_{i_{\sigma(2r)}})
    $$
   \item
    Let $A$ be a skew-symmetric matrix and $\sigma$ a permutation of type I, II
    or III as above.
    Then
    $$
    a_{i_{\sigma(1)},i_{\sigma(2)}}    a_{i_{\sigma(3)},i_{\sigma(4)}} \dotsm
    a_{i_{\sigma(2r-1)},i_{\sigma(2r)}}
    = -    a_{i_1,i_2}    a_{i_3,i_4} \dotsm    a_{i_{2r-1},i_{2r}}
    $$
  \end{enumerate}
\end{lemm}
\begin{proof}
  Both claims are easily verified for each type separately by inspecting the diagrams in
  Figures~\ref{fig:involutiontypeI}, \ref{fig:involutiontypeII} and
  \ref{fig:involutiontypeIII}.
  On the one hand, the permutations act tracially on the blocks of $\pi$ and on
  the other hand, braces are preserved and keep their order with the unique
  exception of the pivot brace which is reversed.
\end{proof}

\begin{proof}[Proof of Theorem~\ref{prop:znikaniekumulant}]
  We will first prove the equivalence of (i)--(iv); then we show that
  \eqref{main:it:cancel} together with \eqref{main:it:antisym} implies
  \eqref{main:it:sym} and finally that \eqref{main:it:sym} implies  \eqref{main:it:antisym}.

  \eqref{main:it:strongcancel}$\implies{}$\eqref{main:it:cancel} is obvious.

  \eqref{main:it:cancel}$\implies{}$\eqref{main:it:preserve}
  follows directly from  
  Theorem~\ref{twr:niskonczonapodzielnoscEL}.

  \eqref{main:it:preserve}$\implies{}$\eqref{main:it:antisym}.
  Fix $i\in\{1,2,\dots,n\}$,
  let  $\X_i$ be  a semicircular element with mean and variance $1$ and put $ \X_j=0$ for
  $j\ne i$.
  Then $T_n=a_{i,i}\X_i^2$ but $\X_i^2$ is not $\boxplus$-infinitely divisible (see
  \cite{Eisenbaum:2012}), and we conclude that $a_{i,i}=0$.

  Thus we have shown that the diagonal entries vanish.
  To cope with the off-diagonal entries, it is sufficient to prove that
  $a_{1,2}=-a_{2,1}$, the proof for the remaining entries being analogous.
  
  Let now   $\X_1$ and $\X_2$ have semicircular distribution with mean and
  variance one and $a_{1,2}=\alpha+i\beta$. Then a computer
  calculation (using \FriCAS{} \cite{fricas}, code available on request) shows that the free
  cumulants of the quadratic form $T_2=a_{1,2} X_1X_2+a_{2,1}X_2X_1$ are
  \begin{align*}
    K_1(T_2) &=2 \alpha
    &
      K_2(T_2) &= 2 \beta^2+10\alpha^2
    \\
      K_3(T_2) &=  24\alpha^3
    &
    K_4(T_2) &=  2\beta^4+4\alpha^2\beta^2+66\alpha^4
    \\
      K_5(T_2) &=  160\alpha^5
    &
      K_6(T_2) & =  2  \beta^6+6  \alpha^2  \beta^4+6  \alpha^4  \beta^2+386 \alpha^6
    \\                 
      K_7(T_2) & =  896\alpha^7
    &
      K_8(T_2) & =  2  \beta^8+8  \alpha^2  \beta^6+12  \alpha^4  \beta^4+8
                 \alpha^6  \beta^2+2050  \alpha^8
    .
  \end{align*}
  To show that this law is not infinitely divisible, it suffices
  to disprove conditional positive definiteness of the cumulant sequence
  \eqref{eq:condposdef}.
  To this end we compute a few Hankel determinants
    \begin{equation}
    h_n =
    \det 
    \left[
      K_{i+j}(T_2)
    \right]_{1\leq i,j\leq n}
  \end{equation}
  and obtain
  \begin{align*}
    h_2&=    4  ( \beta^6+7  \alpha^2  \beta^4+43  \alpha^4  \beta^2+21    \alpha^6 )\\
    h_3&=    32  \alpha^2  ( \beta^2+\alpha^2 ) ( \beta^8 -12  \alpha^2  \beta^6+2  \alpha^4  \beta^4 -52  \alpha^6  \beta^2 -131  \alpha^8 )\\
    h_4&=-256  \alpha^6  ( \beta^2 -3  \alpha^2 )^4  ( \beta^2+\alpha^2)^3 
  \end{align*}
  Thus the fourth determinant is negative unless $\beta=\pm\alpha\sqrt{3}$; in
  that case however $h_3=  - 65536 \alpha^{12}$ and we conclude that $\Re a_{1,2}=\Re
  a_{2,1}=\alpha=0$ and consequently  $\Im a_{1,2}=\Im  a_{2,1}=\beta=0$ as well.

  \eqref{main:it:antisym}$\implies{}$\eqref{main:it:strongcancel}.
  Suppose that $A$ is skew-symmetric. 
  We apply again the product formula from Theorem~\ref{thm:krawczyk} and obtain
  \begin{align*}
   K_r(T_n)
 &= \sum_{i_1,i_2,\dots,i_{2r}\in[n]}
      K_r(a_{i_{1},i_{2}}\X_{i_{1}}\X_{i_{2}},
         a_{i_{3},i_{4}}\X_{i_{3}}\X_{i_{4}},
         \dots,
         a_{i_{2r-1},i_{2r}}\X_{i_{2r-1}}\X_{i_{2r}})
  \\
           &= \sum_{i_1,i_2,\dots,i_{2r}\in[n]}
               \sum_{\substack{\pi\in \NC(2r)\\  
  \pi\vee\makeatletter{}\begin{tikzpicture}[inner sep=0pt,scale=0.04]
\draw (2,0)--(2,4.5);
\draw (8,0)--(8,4.5);
\draw (14,0)--(14,4.5);
\draw (20,0)--(20,4.5);
\node at (30,2){$\cdots$};
\draw (38,0)--(38,4.5);
\draw (44,0)--(44,4.5);
\draw (2,4.5)--(8,4.5);
\draw (14,4.5)--(20,4.5);
\draw (38,4.5)--(44,4.5);
\end{tikzpicture}
 =\hat{1}_{2r}\\
    }
    }
                 a_{i_{1},i_{2}}      a_{i_{3},i_{4}} \dotsm       a_{i_{2r-1},i_{2r}}K_\pi(X_{i_{1}},X_{i_{2}},\dots,X_{i_{2r}})\\
           &= \sum_{
  \substack{
\pi\in \NC(2r)\\  
  \pi\vee\makeatletter{}\begin{tikzpicture}[inner sep=0pt,scale=0.04]
\draw (2,0)--(2,4.5);
\draw (8,0)--(8,4.5);
\draw (14,0)--(14,4.5);
\draw (20,0)--(20,4.5);
\node at (30,2){$\cdots$};
\draw (38,0)--(38,4.5);
\draw (44,0)--(44,4.5);
\draw (2,4.5)--(8,4.5);
\draw (14,4.5)--(20,4.5);
\draw (38,4.5)--(44,4.5);
\end{tikzpicture}
 =\hat{1}_{2r}\\
       }
  }
  \sum_{i_1,i_2,\dots,i_{2r}\in [n]}
  a_{i_{1},i_{2}}      a_{i_{3},i_{4}} \dotsm       a_{i_{2r-1},i_{2r}}K_\pi(X_{i_{1}},X_{i_{2}},\dots,X_{i_{2r}})
  \\
           &= \sum_{
 \pi\in \CompEven{2r}
  }
  \sum_{i_1,i_2,\dots,i_{2r}\in [n]}
   a_{i_{1},i_{2}}      a_{i_{3},i_{4}} \dotsm
             a_{i_{2r-1},i_{2r}}K_\pi(X_{i_{1}},X_{i_{2}},\dots,X_{i_{2r}})
    \\ &\phantom{=}+ \sum_{
  \pi\in \CompOdd{2r}
  }
  \sum_{i_1,i_2,\dots,i_{2r}\in [n]}
   a_{i_{1},i_{2}}      a_{i_{3},i_{4}} \dotsm       a_{i_{2r-1},i_{2r}}K_\pi(X_{i_{1}},X_{i_{2}},\dots,X_{i_{2r}}).
 \end{align*}
 We claim that in this decomposition the second sum cancels.
 To see this, we observe that the involution constructed in
 section~\ref{ssec:involution}  extends to an involution
 \begin{align*}
   [n]^{2r}\times \CompOdd{2r} &\to   [n]^{2r}\times \CompOdd{2r} \\
   (i_1,i_2,\dots,i_{2r};\pi) &\mapsto (i_{\sigma_\pi(1)},i_{\sigma_\pi(2)},\dots,i_{\sigma_\pi(2r)};\sigma_\pi\cdot\pi)
 \end{align*}
 where $\sigma_\pi$ is the permutation constructed
  according to the type I/II/III of $\pi$.
 Therefore we have
 \begin{multline*}
   \sum_{\pi\in \CompOdd{2r}  }
    \sum_{i_1,i_2,\dots,i_{2r}\in [n]}
   a_{i_{1},i_{2}}      a_{i_{3},i_{4}} \dotsm
   a_{i_{2r-1},i_{2r}}K_\pi(X_{i_{1}},X_{i_{2}},\dots,X_{i_{2r}})
   \\
   =
   \sum_{\pi\in \CompOdd{2r}  }
    \sum_{i_1,i_2,\dots,i_{2r}\in [n]}
   a_{i_{\sigma_\pi(1)},i_{\sigma_\pi(2)}}      a_{i_{\sigma_\pi(3)},i_{\sigma_\pi(4)}} \dotsm       a_{i_{\sigma_\pi(2r-1)},i_{\sigma_\pi(2r)}}K_{\sigma_\pi\cdot\pi}(X_{i_{\sigma_\pi(1)}},X_{i_{\sigma_\pi(2)}},\dots,X_{i_{\sigma_\pi(2r)}})
 \end{multline*}
 
 Now by Lemma~\ref{lem:involution} the effect on a term is
 \begin{multline*}
   a_{i_{\sigma(1)},i_{\sigma(2)}}    a_{i_{\sigma(3)},i_{\sigma(4)}} \dotsm
   a_{i_{\sigma(2r-1)},i_{\sigma(2r)}}
   K_{\sigma\cdot\pi}(X_{i_{\sigma(1)}},X_{i_{\sigma(2)}},\dots,X_{i_{\sigma(2r)}})
   \\
   =
   -    a_{i_1,i_2}    a_{i_3,i_4} \dotsm    a_{i_{2r-1},i_{2r}}
        K_\pi(X_1,X_2,\dots,X_{2r})
 \end{multline*}

 and therefore the sum vanishes. 
 
 This concludes the first circle of implications;
 let us now turn to the remaining ones.
 
  \eqref{main:it:cancel}+\eqref{main:it:antisym}$\implies{}$\eqref{main:it:sym}.
  We expand the product formula
  \eqref{twr:produktargumentow} and obtain
  \begin{align*}
   K_r(T_n)
    &=
      \sum_{i_1,i_2,\dots,i_{2r}\in[n]}
        a_{i_{1},i_{2}}a_{i_{3},i_{4}}\dotsm a_{i_{2r-1},i_{2r}}
        K_r(\X_{i_{1}}\X_{i_{2}}, \X_{i_{3}}\X_{i_{4}},\dots,\X_{i_{2r-1}}\X_{i_{2r}})
      \\
    &= \sum_{i_1,i_2,\dots,i_{2r}\in[n]}
      \sum_{\substack{\pi\in \NC(2r) \\ \pi\vee\makeatletter{}\begin{tikzpicture}[inner sep=0pt,scale=0.04]
\draw (2,0)--(2,4.5);
\draw (8,0)--(8,4.5);
\draw (14,0)--(14,4.5);
\draw (20,0)--(20,4.5);
\node at (30,2){$\cdots$};
\draw (38,0)--(38,4.5);
\draw (44,0)--(44,4.5);
\draw (2,4.5)--(8,4.5);
\draw (14,4.5)--(20,4.5);
\draw (38,4.5)--(44,4.5);
\end{tikzpicture}
 =\hat{1}_{2r}} }
        a_{i_{1},i_{2}}a_{i_{3},i_{4}}\dotsm a_{i_{2r-1},i_{2r}}
        K_\pi(\X_{i_{1}},\X_{i_{2}}, \X_{i_{3}},\X_{i_{4}},\dots,\X_{i_{2r-1}},\X_{i_{2r}}),
\intertext{now by assumption  \eqref{main:it:cancel} we may omit all odd cumulants from this formula,
    i.e., we can restrict the sum to even partitions}
    &= \sum_{i_1,i_2,\dots,i_{2r}\in[n]}
      \sum_{\substack{\pi\in \NCeven(2r) \\ \pi\vee\makeatletter{}\begin{tikzpicture}[inner sep=0pt,scale=0.04]
\draw (2,0)--(2,4.5);
\draw (8,0)--(8,4.5);
\draw (14,0)--(14,4.5);
\draw (20,0)--(20,4.5);
\node at (30,2){$\cdots$};
\draw (38,0)--(38,4.5);
\draw (44,0)--(44,4.5);
\draw (2,4.5)--(8,4.5);
\draw (14,4.5)--(20,4.5);
\draw (38,4.5)--(44,4.5);
\end{tikzpicture}
 =\hat{1}_{2r}} }
        a_{i_{1},i_{2}}a_{i_{3},i_{4}}\dotsm a_{i_{2r-1},i_{2r}}
        K_\pi(\X_{i_{1}},\X_{i_{2}}, \X_{i_{3}},\X_{i_{4}},\dots,\X_{i_{2r-1}},\X_{i_{2r}})
    \\
\intertext{and by Lemma~\ref{lemm:lematoparzystych} these partitions have a
    special shape}
    &= \sum_{i_1,i_2,\dots,i_{2r}\in[n]}
      \sum_{\substack{\pi\in \NCeven(2r) \\ \pi\geq\pispecial{r}}}
        a_{i_{1},i_{2}}a_{i_{3},i_{4}}\dotsm a_{i_{2r-1},i_{2r}}
        K_\pi(\X_{i_{1}},\X_{i_{2}}, \X_{i_{3}},\X_{i_{4}},\dots,\X_{i_{2r-1}},\X_{i_{2r}})
    \\
\intertext{and by freeness we can impose the condition $i_{2k}=i_{2k+1}$ on the indices}
    &= \sum_{i_1,i_2,\dots,i_{r}\in[n]}
      \sum_{\substack{\pi\in \NCeven(2r) \\ \pi\geq\pispecial{r}}}
        a_{i_{1},i_{2}}a_{i_{2},i_{3}}\dotsm a_{i_{r},i_{1}}
        K_\pi(\X_{i_{1}},\X_{i_{2}}, \X_{i_{2}},\X_{i_{3}},\dots,\X_{i_{r}},\X_{i_{1}})
    \\
    \intertext{next we apply the mirror permutation $i_j\to i_{r+1-j}$ to the indices;
    this fixes $i_1$ and mirrors the remaining ones}
    &= \sum_{i_1,i_2,\dots,i_{r}\in[n]}
      \sum_{\substack{\pi\in \NCeven(2r) \\ \pi\geq\pispecial{r}}}
        a_{i_{1},i_{r}}a_{i_{r},i_{r-1}}\dotsm a_{i_{2},i_{1}}
        K_\pi(\X_{i_{1}},\X_{i_{r}}, \X_{i_{r}},\X_{i_{r-1}},\dots,\X_{i_{2}},\X_{i_{1}})
    \\
    &= \sum_{i_1,i_2,\dots,i_{r}\in[n]}
    (-1)^ra_{i_{r},i_{1}}a_{i_{r-1},i_{r}}\dotsm a_{i_{1},i_{2}}      
        K_r(\X_{i_{1}}\X_{i_{r}}, \X_{i_{r}}\X_{i_{r-1}},\dots,\X_{i_{2}}\X_{i_{1}})
    \\
    \intertext{where we used assumption \eqref{main:it:antisym} that the matrix $A$ is skew symmetric.
        Now we apply Lemma~\ref{lem:permutedcumulants} (note that the random variables $X_i$
    are free and self-adjoint and therefore the cumulants are real valued) and obtain}
    &= (-1)^r \sum_{i_1,i_2,\dots,i_{r}\in[n]}
      a_{i_{1},i_{2}} a_{i_{3},i_{4}}\dotsm a_{i_{r-1},i_{r}}
        K_r(\X_{i_{1}}\X_{i_{2}}, \X_{i_{2}}\X_{i_{3}},\dots,\X_{i_{r}}\X_{i_{1}})
 \end{align*}
 which implies $K_r(T_n)=0$ for odd $r$.

  \eqref{main:it:sym}$\implies{}$\eqref{main:it:antisym} is easily verified.
  Fix $i$, put $X_i=I$ and $X_j=0$ for $j\ne i$. Then $T_n=a_{ii}^2I$ and property
  \eqref{main:it:sym} implies that $a_{ii}=0$. To cope with the off-diagonal
  terms,
  put $X_1=X_2=I$, then $T= a_{1,2}X_1X_2+a_{2,1}X_2X_1=2\Re a_{1,2}I$ is odd
  and therefore $\Re a_{1,2}=0$. This implies that $A=-A^T$.

\end{proof}

\begin{Rem}  The traces of the odd powers of a skew-symmetric matrix are zero. This fact can
  be generalized as follows.
  A selfadjoint matrix $A\in M_n(\IC)$ is skew-symmetric if and only if for
  every $\pi\in\NC(r)$ where $r$ is  odd we have $\Tr(\ED[\pi](A))=0$.
 \end{Rem}

\subsection{Distributions of quadratic forms}
In \cite[Theorem 1.2]{NicaSpeicher:1998} the authors 
provide an analytic description of the $R$-transform of free commutators in
terms of the  combinatorial convolution $\boxstar$ of the 
even cumulant transforms. To be specific, if $X$ is free from $Y$, then
$$
\Rtrans_{i(XY-YX)}(z)=2(\Rtranseven{X}\boxstar\Rtranseven{Y}\boxstar\zeta_1)(z^2).
$$ 
where by  $\Rtranseven{Z}(z):=\sum_{n=1}^\infty K_{2n}(Z)z^n$ we denote the generating
function of the even free cumulants of $Z\in \A_{sa}$.
An important ingredient in  proof of the preceding results is the notion of
$R$-diagonality. 
An $R$-diagonal pair is a pair of random variables $A$ and $B$
such that  all cumulants vanish with the exception of the alternating ones, i.e.,
those of the form  $K_{2r}(A,B,A,B,\dots,A,B)$ and $K_{2r}(B,A,B,A,\dots,B,A)$.
It turns out that for free even elements $X$ and $Y$ the products
$XY$ and $YX$ form an $R$-diagonal pair and therefore the moments
of the commutator $i(XY-YX)$ are computable.

This observation however is specific to the commutator.
The result below gives an alternative combinatorial description and holds for
arbitrary quadratic forms in even elements. 
For this purpose we define for    $A \in M_n(\IC)$ the generating function 
\begin{equation*}
  f_{ A}(z_1,\dots,z_n):=\sum_{r=1}^\infty\sum_{ i_1,\dots,i_r=1}^n\Tr(AE_{i_1}AE_{i_2}\dots AE_{i_r})z_{i_1}\dots z_{i_r}.
\end{equation*}
and recall that the \emph{Hadamard product} of two formal power series
$f(z) = \sum a_nz^n$ and $g(z)=\sum b_n z^n$
is defined as
$$
f(z) \odot g(z)  = \sum a_nb_nz^n
.
$$
\begin{theo}
  Let $\X_1, \X_2,\dots, \X_n\in \A_{sa}$ be a free family  
  of even random variables.
  Then for any  selfadjoint matrix  $A=[a_{i,j}]_{i,j=1}^n\in M_n(\IC)$ 
  the $\Rtrans$-transform of the quadratic form
  $T_n=\sum_{i,j}^na_{i,j}\X_i\X_j$ can be represented as the Hadamard product
  $$\Rtrans_{T_n}(z)=\big[f_A \odot\big((\Rtranseven{\X_{1}}+\dots+\Rtranseven{ \X_{n}})\boxstar\zeta_n\big)\big](\underbrace{z,\dots,z }_{n}).$$
\end{theo}
\begin{proof}
Using  Proposition \ref{prop:CykliczneVariancja} (\ref{it:cyclic3}), we have 
\begin{align*}  
 K_r(T_n)&=
\sum_{i_1,\dots,i_r\in[n]} 
            \Tr(AE_{i_1}AE_{i_2}\dots AE_{i_r})\,
            \sum_{\substack{ \pi\in \NCeven(2r)\\ 
                \pi \vee \onetwo{r}=\hat{1}_{2r}}}
            K_\pi(X_{i_r},X_{i_1},X_{i_1},X_{i_2},\dots,X_{i_{r-1}},X_{i_r}), 
\intertext{and by traciality of cumulants this is} 
&=\sum_{i_1,\dots,i_r\in[n]} 
            \Tr(AE_{i_1}AE_{i_2}\dots AE_{i_r})\,
            \sum_{\substack{ \pi\in \NCeven(2r)\\ 
               \pi \geq \makeatletter{}\begin{tikzpicture}[inner sep=0pt,scale=0.04]
\draw (2,0)--(2,4.5);
\draw (8,0)--(8,4.5);
\draw (14,0)--(14,4.5);
\draw (20,0)--(20,4.5);
\node at (30,2){$\cdots$};
\draw (38,0)--(38,4.5);
\draw (44,0)--(44,4.5);
\draw (2,4.5)--(8,4.5);
\draw (14,4.5)--(20,4.5);
\draw (38,4.5)--(44,4.5);
\end{tikzpicture}
 }}
            K_\pi(X_{i_1},X_{i_1},X_{i_2},\dots,X_{i_{r-1}},X_{i_r}, X_{i_r}).     
  \intertext{
  Now let $\tilde{X}_i$ be the formal random variable obtained from $X_i$
  by skipping all odd cumulants, i.e., $K_n(\tilde{X}_i)=K_{2n}(X_i)$, then we
  can use the isomorphism from Lemma~\ref{lemm:lematoparzystych} and continue
} 
 &= \sum_{ i_1,\dots,i_r\in[n]}
    \Tr(AE_{i_1}AE_{i_2}\dots AE_{i_r})
    \,
    \tau (\tilde{X}_{i_1}\tilde{X}_{i_2}\dotsm\tilde{X}_{i_r}).
\intertext{Finally by \cite[Proposition 17.4]{NicaSpeicher:2006}, we can write
   this as}
    &\sum_{ i_1,\dots,i_r\in[n]}\Tr(AE_{i_1}AE_{i_2}\dots AE_{i_r})\Cf_{(i_1,\dots,i_r)}\big((\Rtranseven{\X_{1}}+\dots+\Rtranseven{ \X_{n}})\boxstar \zeta_n\big)
                        \end{align*}
which finishes the proof.
\end{proof}

\subsection{Preservation of free infinite divisibility for higher order polynomials}

There are many higher order polynomials which cancel odd
cumulants and preserve infinite divisibility, take for example higher free
commutators like $[[X,Y],Z]$ or 
$[[X,Y],[[A,B],Z]]$. Similarly, take a skew-symmetric quadratic form 
$T_n=\sum_{i,j=1}^na_{i,j}\X_i\X_j$
as in Theorem~\ref{prop:znikaniekumulant},
then  $T_n$ is symmetric and we conclude from \cite[Theorem~2.2]{ArizmendiHasebeSakuma:2013}
that $T_n^2$ preserves free infinite divisibility.

The purpose of
this subsection is to produce   higher order polynomials which preserve free
infinite divisibility but don't exhibit the cancellation phenomenon;
we do not know however whether the reverse implication is true.

For the concept of \emph{free regular distributions} we refer to
\cite{ArizmendiHasebeSakuma:2013}.
\begin{prop}
Let  $X,Y, Z$ be free random variables such that $X$ and $Y$ are freely
infinitely divisible. Then
the selfadjoint element
$[X,Y]Z[X,Y]$ has
compound free Poisson distribution of rate $1$ with jump distribution $\mu_Z\boxtimes\sigma$ for some free regular distribution $\sigma$.
Consequently it is freely infinitely divisible and the odd cumulants of $Z$ do not cancel.
\end{prop}
\begin{proof}
First recall that
\cite[Theorem~2.2]{ArizmendiHasebeSakuma:2013}
asserts that if a random variable has even FID law $\mu$ then the law of its square can be decomposed $\mu^2=m\boxtimes\sigma$ where $m$ is the free Poisson
law of rate $1$ and $\sigma$ is free regular.

Next recall that for any law $\nu$ the law $\nu\boxtimes m$ is the law 
of the free compression with a semicircular random variable and
therefore free Poisson with jump distribution $\nu$.

Now let $\mu$ be the law of  $i[X,Y]$.
It follows from \cite[Theorem~1.2]{NicaSpeicher:1998} and
\cite[Corollary~6.5]{ArizmendiHasebeSakuma:2013} that $\mu$ is both even and
freely infinitely divisible.
The law of $[X,Y]Z[X,Y]$ is $\mu_Z\boxtimes\mu^2=\mu_Z\boxtimes\sigma\boxtimes
m$ and therefore our random variable has compound free Poisson
distribution with rate $1$ and jump distribution $\mu_Z\boxtimes\sigma$.
\end{proof}

\subsection{The generalized tetilla law }
Formulas \eqref{eq:kumulantsamplevariancenotiid}
and \eqref{eq:kumulantsamplevariance} are hard to evaluate in general.
There are however two settings for which the distribution can be
computed explicitly.
The first result is  a kind of central limit theorem for sums of commutators 
which gives rise to the \emph{free tangent law} and which will appear 
in a separate paper  \cite{EjsmontLehner:2020:tangent}.

The second result is presented in this section and concerns sums of 
commutators of semicircular elements for which the sum
\eqref{eq:kumulantsamplevariancenotiid} simplifies 
considerably.

Motivated by \cite{DeyaNourdin:2012} we propose the following definition.
\begin{defi}
  Let $\X_1, \X_2,\dots, \X_n\in \A_{sa}$ with $n \geq 2$  be a free  family of semicircular random variables of variance one.
The law of the random variable 
$\sum_{k,j=1}^ni(\X_k\X_j-\X_j\X_k)$  is called the \emph{generalized tetilla
  law} with $n$ degrees of freedom. We denote this distribution by $\T_n$.
\end{defi}
The \emph{tangent numbers}
\begin{equation}
  \label{eq:tangentnumbers}
  T_k=(-1)^{k+1}\frac{4^{k}(4^{k}-1){B_{2k}}}{2k}
\end{equation}
for $k\in\N $ are the Taylor coefficients of the tangent function
$$
\tan z = \sum_{n=1}^\infty T_n\frac{z^n}{n!} = z + \frac{2}{3!} z^3 +  \frac{16}{5!}z^5 +\frac{272}{7!}z^7 + \dotsm,
$$
see \cite[Page 287]{GrahamKnuthPatashnik:1994}).

On the other hand let us denote by $A_n^{(k)}$ the
\emph{arctangent numbers} 
(see \cite[p.~260]{Comtet} or 
\cite{Cvijovic:2011:higher}) 
defined by their exponential generating function
\begin{equation}
  \label{eq:def-arctannumbers}
  \frac{(\arctan z)^k}{k!} = \sum_{n=k}^\infty \frac{A_n^{(k)}}{n!}z^n;
\end{equation}
\begin{prop} \label{prop:tetilla}
The generalized tetilla law with $n$ degrees of freedom has the following properties.
 \begin{enumerate}[(i)]
  \item It is $\boxplus$-infinitely divisible with discrete L\'evy measure
   $\nu=\delta_{\cot\left(\frac{\pi}{2n}\right) }+\dots+\delta_{\cot\left(\frac{\pi}{2n}+\frac{n-1}{n}\pi\right) }$;
  \item It is symmetric and its even cumulants are
   \begin{equation}
     \label{eq:K2m=tangent}
   K_{2m}
    = (-1)^{m}n
    + \frac{1}{(2m-1)!}   \sum_{k=1}^m n^{2k} A_{2m}^{(2k)}\,T_{2k-1}
   \end{equation}
   where by $T_k$ and $A_{m}^{(k)}$ we denote the tangent and  arctangent numbers, respectively.
\item The $R$-transform is equal to 
$$R_{\T_n}(z)=\frac{n\tan(n\arctan z)-nz}{1+z^2}$$
\end{enumerate} 
\end{prop}
\begin{proof}
Symmetry and  $\boxplus$-infinite divisibility  follow  directly from
 Theorem \ref{prop:znikaniekumulant}.
Since  $X_i$  are random variables of variance one, 
 we can compute the cumulants by using formula \eqref{eq:cumQnsemi} 
and evaluate to
$$K_r\big(i \sum_{k,j=1}^n(\X_k\X_j-\X_j\X_k)\big)=\Tr(A_n^r)
\text{ where } 
A_n=\left[\begin{smallmatrix}
    0 &i \\
    -i& 0
    \end{smallmatrix}\right]_n. 
$$
The eigenvalues and characteristic polynomial of the matrix $A_n$ were computed in \cite[Lemma 2.1, $\alpha=\pi/2$]{EjsmontLehner:2019:cotangent}   and they  are
$
\lambda_k=\cot\left(\frac{\pi}{2n}+\frac{k\pi}{n}\right)$ for $k\in\{0,\dots,n-1\} 
$ (including repeated eigenvalues), and 
\begin{equation}
  \label{eq:pnlambda}
\chi_n(\lambda) 
= \frac{(\lambda-i)^n+(\lambda+i)^n}{2}.
\end{equation}
Hence the odd cumulants vanish and the even cumulants
are equal to the cotangent sums
\begin{align*}
  K_{2m}\big(\sum_{k,j=1}^ni(\X_k\X_j-\X_j\X_k)\big)
  &=\sum_{k=0}^{n-1} \cot^{2m}\left(\frac{\pi}{2n}+\frac{k}{n}\pi\right)\\
\end{align*}
which were evaluated explicitly in
\cite[Corollary 6.4]{EjsmontLehner:2019:cotangent}
and the result is \eqref{eq:K2m=tangent}.

Once having realized $\lambda_k$ as roots of a polynomial,
it is easy to write down the generating function 
as a logarithmic derivative. Indeed, let 
\begin{align*}
  g_n(z) 
  &= \sum_{k=0}^{n-1} \frac{1}{z-\lambda_k}\\
  &= \frac{\chi_n'(z)}{\chi_n(z)}   \\
  &= n\frac{
    (z-i)^{n-1}
    +(z+i)^{n-1}
  }{
    (z-i)^n
    +(z+i)^n
  }
\end{align*}
then the ordinary generating function is
\begin{align*}
 F_n(z)&=\sum_{m=0}^\infty \sum_{k=0}^{n-1}\cot^m\frac{\pi/2+k\pi}{n}\,z^m\\  
 &= \frac{1}{z}g_n\left(\frac{1}{z}\right)  \\
  &= n\frac{
    (1-iz)^{n-1}
   +(1+iz)^{n-1}
  }{
    (1-iz)^n
   +(1+iz)^n
    }
  \\
 &= \frac{n}{1+z^2}
   \frac{
    (1-iz)^n(1+iz)
    +(1+iz)^n(1-iz)
  }{
    (1-iz)^n
    +(1+iz)^n
    }
  \\
 &= \frac{n}{1+z^2}
   \left(
   1+iz
   \frac{
    (1-iz)^n
    -(1+iz)^n
  }{
    (1-iz)^n
    +(1+iz)^n
    }
   \right)
  \\
 &= \frac{n(1+z\tan(n\arctan z))}{1+z^2}
\end{align*}
where in the last step we used the well known formula \cite[item 16]{HAKMEM}
\begin{equation}
\label{eq:tannatanz}
\tan(n\arctan z) = i\frac{(1-iz)^n-(1+iz)^n}{(1-iz)^n+(1+iz)^n}
;
\end{equation}
Finally we use the relation $R_{\T_n}(z)= (F_n(z)-n)/z$, which gives $R$-transform. 
\end{proof}

\subsection{Free skew-symmetric laws}

\begin{defi} \label{freeskewsymmetric}
  Let $\X_1, \X_2,\dots, \X_n\in \A_{sa}$   be a free  family of semicircular random variables with variance one and 
  $A=[a_{i,j}]_{i,j=1}^n\in M_n(\IC)$ be a selfadjoint matrix such that $A=-A^T.$
The law of the random variable 
$\sum_{k,j=1}^na_{k,j}(\X_k\X_j-\X_j\X_k)$  is called the \emph{free skew-symmetric distribution} with matrix $A$. 
\end{defi}

\begin{prop} \label{decopositionskewsymmetric}
A distribution  $\mu$ is free skew-symmetric with matrix $A$ if and only if 
 $\mu$   can be decomposed as a free convolution of rescaled tetilla distributions
$$\mu=D_{\lambda_1} (\T_2)\boxplus \dots \boxplus {D_{\lambda_{\lfloor n/2 \rfloor}}}(\T_2),$$
where the scale parameters  $\lambda_i$ are the positive  eigenvalues of $A$ and  the dilation $D_r$ is defined as $D_r(\nu)(A)=\nu(A/r)$ if $r\neq 0$ and $D_0(\nu)(A)=\delta_0$. ${\lfloor}\cdot {\rfloor}$ is the floor function which
rounds down to the nearest integer.
\end{prop}
\begin{proof} 
    Let $A$ be a selfadjoint skew-symmetric matrix and
  and $\mu$ the corresponding distribution.
  Assume first that  $n$ is even.
  Recall that $iA$ is a real skew-symmetric matrix so the nonzero eigenvalues of this matrix are $\pm i\lambda_1,\dots,\pm i\lambda_{n/2}$.
  It is possible to bring every skew-symmetric matrix to a block diagonal form by
  an orthogonal transformation, see for example \cite{Youla:1961}.
  To be specific, every $n \times n$  real skew-symmetric matrix can be written in the form $iA=Q\Sigma Q^T$  
  where $Q$ is orthogonal and
  \begin{equation}
    \label{eq:skewmatrix}
  \Sigma=
  \begin{bmatrix} 
    0  & \lambda_1 & \dots  &  0 &  0  \\ 
    -\lambda_1 &  0 & \dots & 0 & 0
    \\
    \multicolumn{5}{c}{\dotfill}\\
    0 &  0& \dots & 0 &  \lambda_{n/2}
    \\ 0 & 0&   \dots & -\lambda_{n/2}  &   0
  \end{bmatrix},
  \end{equation}
  for real $\lambda_k$.  
Let $\X_1, \X_2,\dots, \X_n\in \A_{sa}$ be as in Definition \ref{freeskewsymmetric}, then we have 
\begin{align*}
  \sum_{k,j=1}^na_{k,j}(\X_k\X_j-\X_j\X_k)
  &=-i\begin{bmatrix}  X_1 & \dots & \X_n \end{bmatrix}iA \begin{bmatrix}  X_1 & \dots & X_n \end{bmatrix} ^T
  \\
  &=-i\begin{bmatrix}
    X_1 & \dots & X_n
  \end{bmatrix}
  Q\Sigma Q^T \begin{bmatrix} X_1 & \dots & X_n \end{bmatrix}^T.
  \intertext{Now   by \cite[Theorem~3.5]{HiwatashiNagisaYoshida:1999} the
    vector $\begin{bmatrix} Y_1 & \dots & Y_n \end{bmatrix}:=\begin{bmatrix} X_1 & \dots & X_n \end{bmatrix} Q$
   is a   free  family of semicircular random variables with variance one
   and we obtain a linear combination of free tetilla elements
    }
&=\lambda_1i(Y_2Y_{1}-Y_{1}Y_2)+\lambda_2i(Y_4Y_{3}-Y_{3}Y_4)+\dots+\lambda_{n/2}i(Y_nY_{n-1}-Y_{n-1}Y_n).
\end{align*}
For the converse, just pick the matrix $\Sigma$ from \eqref{eq:skewmatrix}.

In the odd-dimensional case the same orthogonal decomposition is true but in this
case $\Sigma$ always has at least one row and column of zeros
which does not contribute.
\end{proof}
\begin{cor} 
  \begin{enumerate}[1.]
  \item  Every free skew-symmetric distribution with system matrix $A$ of odd
   degree $n$ can be represented as  a skew-symmetric distribution of even degree $n-1$.
   In particular the generalized tetilla
   law with $3$ degrees of freedom can be obtained by a dilation from the tetilla law as 
   $$\T_3=D_{\sqrt{3}}(\T_2),$$
   the corresponding eigenvalues being $\cot(\frac{\pi}{6})=\sqrt{3}$, $\cot(\frac{\pi}{2})=0$ and $\cot(\frac{5\pi}{6})=-\sqrt{3}$. 

\item  Every free skew-symmetric distribution $\mu$ is a compound free Poisson distribution. 
Indeed,   $\T_2$ has a compound free Poisson distribution with symmetric  jump distribution
$\frac{1}{2}\delta_{-1}+\frac{1}{2}\delta_{1}$ and rate 2.  From Proposition~\ref{decopositionskewsymmetric}
(with the same designation) we infer  that $\mu$ has free compound Poisson distribution with rate $n$ and symmetric  jump distribution 
$$\frac{1}{n}\delta_{\lambda_1}+\frac{1}{n}\delta_{-\lambda_1}+\dots+\frac{1}{n}\delta_{{-\lambda_{\lfloor n/2 \rfloor}}}+\frac{1}{n}\delta_{{\lambda_{\lfloor n/2 \rfloor}}}.$$ 
Consequently
every compound free Poisson variable with  symmetric  jump distribution supported on a
finite set, with rate $n$ and evenly distributed mass can  be modeled as a linear combination of free commutators. 
\end{enumerate}
\end{cor}

\bibliographystyle{amsplain}

\providecommand{\bysame}{\leavevmode\hbox to3em{\hrulefill}\thinspace}
\providecommand{\MR}{\relax\ifhmode\unskip\space\fi MR }
\providecommand{\MRhref}[2]{  \href{http://www.ams.org/mathscinet-getitem?mr=#1}{#2}
}
\providecommand{\href}[2]{#2}

\end{document}